\documentclass[a4,12pt]{article}

\usepackage[english]{babel}
\usepackage[utf8]{inputenc}

\usepackage[margin=1.13 in, top=1.1 in, bottom= 1.2 in]{geometry}
\usepackage{tocloft}

\usepackage{amsfonts}
\usepackage{mathrsfs}
\usepackage{bbm}
\usepackage{latexsym}
\usepackage{amssymb}
\usepackage{mathtools}

\usepackage{enumitem}
\setlist{nolistsep} 	
\usepackage{amsthm}

\usepackage{xcolor}
\definecolor{Color1}{rgb}{0.0, 0.42, 0.47}
\definecolor{Color2}{rgb}{0.78, 0.11, 0.0}

\usepackage{titlesec}
\titleformat{\section}
  {\large\center\bfseries}
  {\thesection.}{.7em}{}
\titlespacing*{\section}{0pt}{4ex}{2ex}
\titleformat{\subsection}
  {\center\bfseries}
  {\thesubsection.}{.7em}{}
\titlespacing*{\subsection}{0pt}{3.5ex}{1.5ex}
\titleformat{\subsubsection}
  {\center\bfseries}
  {\thesubsubsection.}{.7em}{}
\titlespacing*{\subsubsection}{0pt}{3.5ex}{1.5ex}

\addto\captionsenglish{}

\usepackage{titling}
\usepackage[linktocpage=true]{hyperref}
\usepackage[capitalize]{cleveref}
\hypersetup{citecolor = Color1,colorlinks,
			linkcolor = black,
			urlcolor = Color2}

\newtheoremstyle{plain}{3mm}{3mm}{\slshape}{}{\bfseries}{.}{.5em}{}
\newtheoremstyle{definition}{2mm}{2mm}{}{}{\bfseries}{.}{.5em}{}
\theoremstyle{plain} 
	
\newtheorem{theorem}{Theorem}[section]

\newtheorem{proposition}[theorem]{Proposition}

\newtheorem{lemma}[theorem]{Lemma}
\newtheorem{corollary}[theorem]{Corollary}
\theoremstyle{definition} 
\newtheorem{definition}[theorem]{Definition}
\newtheorem{remark}[theorem]{Remark}
\newtheorem{example}[theorem]{Example}

\usepackage{chngcntr}

\numberwithin{equation}{section}

\allowdisplaybreaks

\newcommand{\Cesaro}{Ces\`{a}ro}

\newcommand{\Oh}{{\mathrm O}}
\newcommand{\oh}{{\mathrm o}}
\newcommand{\N}{\mathbb{N}}
\newcommand{\Z}{\mathbb{Z}}
\newcommand{\R}{\mathbb{R}}
\newcommand{\C}{\mathbb{C}}
\newcommand{\Q}{\mathbb{Q}}

\newcommand{\define}[1]{{\itshape #1}}

\renewcommand{\epsilon}{\varepsilon}
\renewcommand{\leq}{\leqslant}
\renewcommand{\geq}{\geqslant}
\renewcommand{\setminus}{\backslash}

\renewcommand{\P}{\mathbb{P}}
\renewcommand{\subset}{\subseteq}

\renewcommand{\d}{~\mathrm{d}}

\usepackage[normalem]{ulem}

\usepackage[makeroom]{cancel}

\newcommand{\E}{\operatorname{\mathbb{E}}}

\newcommand{\oo}{\infty}
\newcommand{\bin}{\text{bin}}
\newcommand{\floor}[1]{\lfloor #1 \rfloor}

\newcommand{\norm}[1]{\| #1 \|}

\makeatletter
\renewenvironment{abstract}
  {\small
   \begin{center}
     \bfseries Abstract
   \end{center}
   \quotation}
  {\endquotation}
\makeatother

\author{By~~{Vitaly~Bergelson}~~and~~ {Michael Reilly}
\\
and~~{Florian~K.~Richter}}
\date{}
\title{Weighted averages of arithmetic functions and applications to equidistribution and ergodic theory}

\begin{document}

\maketitle

\begin{abstract}
For a wide range of functions $W\colon\N\to\N$, we establish a general result for estimating weighted averages of the form
\[
\E^{W}_{n \le N} f(\vartheta(n))= \frac{1}{W(N)}\sum_{n=1}^N (W(n)-W(n-1))f(\vartheta(n)),
\]
where $f\colon \{1,\ldots,N\}\to\C$ is an arbitrary function, and $\vartheta(n)$ is any arithmetic function that adheres to a certain Gaussian distribution condition.(For instance, one may take $\vartheta(n)=\Omega(n)$, where $\Omega(n)$ counts the number of prime factors of $n$ with multiplicity, or $\vartheta(n)=s_q(p_n)$, where $s_q$ is the sum-of-digits function in base $q$ and $p_n$ denotes the $n$-th prime. Additional natural examples are discussed in the paper.)
Building on our main theorem, we show that if $h(n)$ is a function from a Hardy field with polynomial growth then $(h(\vartheta(n)))_{n\in\N}$ is uniformly distributed mod~$1$ if and only if one of the following (mutually exclusive) conditions is satisfied:
\begin{enumerate}[label=(\roman{enumi}),leftmargin=*]
\item $\lim_{x\to\infty} \frac{|h(x)-p(x)|}{x \log x}=\infty$ for all $p(x)\in \Q[x]$;
\item $\lim_{x\to\oo}\frac{|h(x)-p(x)|}{\sqrt{x}}=\oo$ for each $p(x)\in \Q[x]$ and there exists $q(x)\in \Q[x]$ such that $\lim_{x\to\oo}\frac{|h(x)-q(x)|}{x}<\oo$.
\end{enumerate}
This leads to several novel applications. For example, it follows that $(\Omega(n)^c)_{n\in\N}$ is uniformly distributed mod~$1$ if and only if $c$ is a non-integer greater than $\frac{1}{2}$.
\end{abstract}

\thispagestyle{empty}


\section{Introduction}\label{section:introduction}

Let $\N=\{1,2,3,\ldots\}$ be the set of positive integers.
Given a sequence $W \colon \mathbb{N} \to [0,\infty)$,
we define its \define{discrete derivative} $\Delta W$ by
\[
\Delta W(N) =
\begin{cases}
W(N) - W(N-1), & \text{if } N \ge 2, \\[6pt]
W(1), & \text{if } N = 1.
\end{cases}
\]
We also use second and third order discrete derivatives 
$\Delta^2 W = \Delta(\Delta W)$ and $\Delta^3 W = \Delta(\Delta^2 W)$.

\begin{definition}
Let $\mathscr{W}$ denote the class of functions $W\colon \N\to[0,\infty)$ that are eventually non-decreasing and satisfy $\lim_{N \to \infty} W(N) = \infty$.  
For $W \in \mathscr{W}$ and $f \colon \{1, \ldots, N\} \to \mathbb{C}$, we define the \emph{weighted discrete average}
\begin{equation}\label{eq:equivalence_of_derivatives_0}
\E^{W}_{n \le N} f(n)
= \frac{1}{W(N)} \sum_{n=1}^N \Delta W(n) \, f(n).
\end{equation}
When $W(N) = N$, this reduces to the standard \define{Cesàro average}, which we denote by
\[
\E_{n \le N} f(n) = \frac{1}{N} \sum_{n=1}^N f(n).
\]
\end{definition}

In addition to the averages introduced in \eqref{eq:equivalence_of_derivatives_0}, we consider two averaging schemes whose weights do not arise from a single underlying function: 
the \emph{binomial mean}
\begin{equation}
\label{def_bin}
\E^{\bin}_{n \le N} f(n)
= \frac{1}{2^N} \sum_{n=1}^N \binom{N}{n} f(n),
\end{equation}
and its ``parity-neutral'' variant
\begin{equation}
\label{def_2bin}
\E^{2\bin}_{n \le N} f(n)
= \E^{\bin}_{n \le N} \!\left( \frac{f(2n) + f(2n+1)}{2} \right)
= \frac{1}{2^{N+1}} \sum_{n=1}^{2N} \binom{N}{\big\lfloor \frac{n}{2} \big\rfloor} f(n).
\end{equation}

Let $f\colon \{1,\ldots,N\}\to\C$ be an arbitrary function. The purpose of this paper is to develop a new technique for estimating averages of the form
\begin{equation}
\label{eqn_general_theta_average}
\E^{W}_{n \le N} f(\vartheta(n))
\end{equation}
for an extensive class of arithmetic functions $\vartheta\colon\N\to\N$.

To describe the class of sequences $\vartheta(n)$ to which our method applies, 
consider the probability density function of the Gaussian normal distribution with mean $\mu$ and standard deviation~$\sigma$ given by
\[
g(x,\mu,\sigma) = \frac{1}{\sigma \sqrt{2\pi}} \, e^{-\frac{(x-\mu)^2}{2\sigma^2}},\qquad x\in\R.
\]
Let $\mathscr{L}\subset\mathscr{W}$ denote the class of all sequences $L\in\mathscr{W}$ for which $\Delta L(n)\in\{0,1\}$ for all but finitely many $n\in\N$; this requirement can be interpreted as a discrete analogue of sublinear growth.
The scope of our main result includes all arithmetic functions $\vartheta\colon \N\to\N$ for which there exists $L\in\mathscr{L}$ such that
\begin{equation}\label{gaussian_condition}
\sum_{k\in \N} \bigg|\frac{|\{1\leq n\leq N: \vartheta(n)=k\}|}{N}- g(k,L(N),\sqrt{L(N)})\bigg|= \oh_{N\to\infty}(1).
\end{equation}
This ``Gaussian condition'' roughly says that the distribution of $\vartheta(n)$ is close in variation distance to a normal distribution with mean $L(N)$ and standard deviation $\sqrt{L(N)}$.

There are many types of naturally occurring functions that satisfy this condition. For example, this property is satisfied by various summatory functions arising in number theory: 
\begin{itemize}
\item 
Let
\[\Omega(n)=\sum_{p^k} 1_{p\mid n}\]
denote the number of prime factors of $n$ counted with multiplicity.
Then $\Omega(n)$ satisfies condition \eqref{gaussian_condition} with $L(N) = \lfloor \log\log(N) \rfloor$. This follows from  
\cite[Theorem~II]{Erdos48a} (see~\cite[Lemma~3.4]{LOYD23} for details).
\item
Let
\[
\omega(n)=\sum_{p} 1_{p\mid n}\]
be the number of prime factors of $n$ counted without multiplicities.
Then $\omega(n)$ satisfies \eqref{gaussian_condition} with $L(N) = \lfloor \log\log(N) \rfloor$. This can be derived from \cite[Theorem~I]{Erdos48a}.
\item
The function $\Omega(q_n)$ satisfies \eqref{gaussian_condition} with $L(N) = \lfloor \log\log(N) \rfloor$, where $q_n$ is the $n$-th squarefree number. This follows from \cite[Theorem~III]{Erdos48a}.
\item 
Given an integer $q\geq 2$, let $s_q(n)$ denote the sum of digits function in base $q$, i.e.,
\[
s_q(n)=\sum_{j\geq 0}\epsilon_j(n),\quad\text{where}\quad n=\sum_{j\geq 0} \epsilon_j(n) q^j.
\]
Then $s_q(n)$ satisfies \eqref{gaussian_condition} with $L(N) = \lfloor \log(N) \rfloor$; this can be derived from \cite[Théorème~1.1]{FM05}.
\item 
The function $s_q(p_n)$ satisfies \eqref{gaussian_condition} with $L(N) = \lfloor \log(N) \rfloor$, where $p_n$ is the $n$-th prime number. In fact, a stronger assertion is proved in \cite[Theorem~1.1]{DMR09}.
\end{itemize}
While our main results are stated below for general $\vartheta(n)$ satisfying condition \eqref{gaussian_condition}, it is worth noting that they are already new when $\vartheta(n)$ is taken to be any one of the functions $\Omega(n), \omega(n), \Omega(q_n), s_q(n), s_q(p_n)$.

For technical reasons, in our main theorem we need to restrict our attention to weights in $\mathscr{W}$ that exhibit suitable regularity at infinity. More precisely, we require that the weight function $W$ has the property that
\begin{equation}
\label{eqn_smoothish_weights}
\lim_{N\to\infty} \frac{N\cdot \Delta^2 W(N)}{\Delta W(N)} \text{ exists in } \R\cup \{-\infty,\infty\}.
\end{equation}
This assumption is mild, as many natural weights satisfy \eqref{eqn_smoothish_weights}. 
For instance, if $W$ is a function that belongs to a Hardy field (defined on page~\pageref{hardy_page}) then \eqref{eqn_smoothish_weights} holds and hence $W \in \mathscr{W}^*$ if and only if $W \in \mathscr{W}$.

Henceforth, we use $\mathscr{W}^*$ to denote the subclass of $\mathscr{W}$ that satisfy \eqref{eqn_smoothish_weights}.
The following is our main theorem.

\begin{theorem}
\label{thm_main}
Let $W\in\mathscr{W}^*$ and assume $\vartheta\colon \N\to\N$ satisfies \eqref{gaussian_condition} for some $L\in\mathscr{L}$.
\begin{enumerate}
\item\label{itm_main_1}
Uniformly over all $f\colon \N\to\C$ with $\|f\|_\infty\leq 1$,
\begin{equation}
\label{eq_main_cesaro_scale}
\E_{n\leq N}f(\vartheta(n)) =\E_{n\leq L(N)}^{2\bin}f(n)+o_{N\to\oo}(1).
\end{equation}
\item
If $\lim_{N\to\oo}\frac{\log(W(N))}{\log(N)}=0$ then uniformly over all $f\colon \N\to\C$ with $\|f\|_\infty\leq 1$,
\begin{equation}
\label{eq_main_log_scale}
\E_{n\leq N}^W f(\vartheta(n)) =\E_{n\leq N}^W\E_{k\leq L(n)}^{2\bin}f(k)+o_{N\to\oo}(1).
\end{equation}
\item\label{itm_main_3}
If $\lim_{N\to\oo}\frac{\log(W(N))}{N}=0$ and $\lim_{N\to\oo}\frac{\log(W\circ L)(N)}{\log(N)}=0$ then uniformly over all $f\colon \N\to\C$ with $\|f\|_\infty\leq 1$,
\begin{equation}
\label{eq_main_loglog_scale}
\E_{n\leq N}^{W\circ L} f(\vartheta(n))  = \E^{W}_{n\leq L(N)}f(n)+o_{N\to\oo}(1).
\end{equation}
\end{enumerate}
\end{theorem}

Let us point out some interesting consequences of \cref{thm_main} to illustrate its usefulness.
Applied to $\vartheta(n) = \Omega(n)$ and $W(N) = N$, part~\ref{itm_main_1} gives
that the \Cesaro{} average of $f(\Omega(n))$ is asymptotically equal to the parity-neutral binomial mean of $f(n)$, i.e.,
\begin{equation}
\label{eqn_Cesaro_for_Omega}
\E_{n\leq N}f(\Omega(n)) =\E_{n\leq \lfloor\log\log N\rfloor}^{2\bin}f(n)+o_{N\to\oo}(1).
\end{equation}
Moreover, part~\ref{itm_main_3} shows that the
double-logarithmic average of $f(\Omega(n))$ corresponds to the \Cesaro{} average of $f(n)$,
\begin{equation}
\label{eqn_double-log-averages_for_Omega}
\E_{n\leq N}^{\log\log} f(\Omega(n))  = \E_{n\leq \lfloor\log\log N\rfloor}f(n)+o_{N\to\oo}(1).
\end{equation}
Analogous formulas hold when $\Omega(n)$ is replaced by either $\omega(n)$ or $\Omega(q_n)$.

In a similar vein, when $\vartheta(n)=s_q(n)$, then part~\ref{itm_main_1} of \cref{thm_main} yields
\begin{equation}
\label{eqn_Cesaro_for_sumofdigits}
\E_{n\leq N}f(s_q(n)) =\E_{n\leq \lfloor\log N\rfloor}^{2\bin}f(n)+o_{N\to\oo}(1),
\end{equation}
and part~\ref{itm_main_3} gives
\begin{equation}
\label{eqn_log-averages_for_sumofdigits}
\E_{n\leq N}^{\log} f(s_q(n))  = \E_{n\leq \lfloor\log N\rfloor}f(n)+o_{N\to\oo}(1).
\end{equation}
The same formulas apply when $s_q(n)$ is replaced by $s_q(p_n)$.

As an application of \cref{thm_main}, we obtain new results on the equidistribution of sequences modulo $1$ and new results in ergodic theory on the convergence of ergodic averages along arithmetic functions. These are described in the following two subsections.

\subsection*{Applications to equidistribution}

We say that two functions $f\colon [a,\infty)\to\R$ and $g\colon [b,\infty)\to\R$ are \define{eventually identical} if there exists $c\geq \max\{a,b\}$ such that $f(x)=g(x)$ for all $x\in [c,\infty)$. This yields a natural equivalence relation on the set of all real-valued continuous functions which are defined for all sufficiently large real arguments.
A \define{germ at $\infty$} of real-valued functions is an equivalence class under this relation.
Note that the operations of pointwise addition and pointwise multiplication of real-valued continuous functions extend in a natural way to germs at $\infty$. Under these operations, the set of all germs at $\infty$ forms a commutative ring.

A \define{Hardy field} $\mathcal{H}$ is any subfield of this ring that is closed under differentiation, in the sense that if the germ at $\infty$ of a differentiable function belongs to $\mathcal{H}$ then so does the germ at $\infty$ of its derivative.\label{hardy_page}

Typical examples of Hardy fields include the field of rational functions, and the field of logarithmico-exponential functions (i.e., the smallest field closed under compositions and containing all polynomials, $\log(x)$, and $\exp(x)$).
By abuse of language, we say a function $f\colon [a,\infty)\to\R$ belongs to a Hardy field if its germ at $\infty$ belongs to a Hardy field. Examples include functions such as $x^c$ for $c\in\R$, $x\log(x)$, or $\exp(\sqrt{\log x})$).
It is a classical fact that functions belonging to a Hardy field are eventually monotone. In particular, this means that highly oscillatory functions such as $\sin(x)$ do not belong to any Hardy field.
For more information on Hardy fields, see \cite{Boshernitzan81,Boshernitzan82,Boshernitzan94} or \cite[Section~2]{Frantzikinakis09}.

We are now ready to formulate one of the main applications of \cref{thm_main}.

\begin{theorem}
\label{thm_ud_of_Hardy_functions_along_Omega_Cesaro}
Let $h$ be a Hardy field function with polynomial growth (i.e., there exist $c,d\geq 1$ such that $|h(x)|\leq x^d$ for all $x\in [c,\infty)$).
Assume $\vartheta\colon \N\to\N$ satisfies \eqref{gaussian_condition} for some $L\in\mathscr{L}$.
The following are equivalent:
\begin{enumerate}[label=(\roman{enumi}),ref=(\roman{enumi}),leftmargin=*]
\item
The sequence $(h(\vartheta(n)))_{n\in \N}$ is uniformly distributed mod~$1$.
\item 
One of the following two (mutually exclusive) conditions is satisfied:
\begin{enumerate}[label=(\alph{enumii}),ref=(\alph{enumii}),leftmargin=*]
\item
\label{itm_ud_of_Hardy_functions_along_Omega_Cesaro_ii_a}
$\lim_{x\to\infty} \frac{|h(x)-p(x)|}{x \log x}=\infty$ for all $p(x)\in \Q[x]$;
\item
\label{itm_ud_of_Hardy_functions_along_Omega_Cesaro_ii_b}
$\lim_{x\to\oo}\frac{|h(x)-p(x)|}{\sqrt{x}}=\oo$ for each $p(x)\in \Q[x]$ and there exists $q(x)\in \Q[x]$ such that $\lim_{x\to\oo}\frac{|h(x)-q(x)|}{x}<\oo$.
\end{enumerate}
\end{enumerate}
\end{theorem}

By L'Hospital's rule, condition \ref{itm_ud_of_Hardy_functions_along_Omega_Cesaro_ii_a} in Theorem~\ref{thm_ud_of_Hardy_functions_along_Omega_Cesaro} is equivalent to the assertion
\[
\lim_{x\to\infty} \frac{|h'(x)-p(x)|}{\log x}=\infty,\quad \forall p(x)\in \Q[x],
\]
where $h'$ denotes the derivative of $h$. In light of Boshernitzan's theorem \cite[Theorem 1.3]{Boshernitzan94}, this is in turn equivalent to $h'(n)$ being uniformly distributed mod~1.
This leads us to the following corollary.

\begin{corollary}
\label{cor_ud_of_Hardy_functions_along_Omega_Cesaro_2}
Let $h$ be a function from a Hardy field with polynomial growth. If $h'(n)$ is uniformly distributed mod~$1$ then $h(\Omega(n))$ is uniformly distributed mod~$1$.
The same applies to the sequences $h(\omega(n))$, $h(\Omega(q_n))$, $h(s_q(n))$, and $h(s_q(p_n))$.
\end{corollary}

The reverse implication in \cref{cor_ud_of_Hardy_functions_along_Omega_Cesaro_2} does not hold. For example, if $h(x)=x^{\frac{2}{3}}$ then $h'(n)=\frac{2}{3 \sqrt[3]{n}}$ is not uniformly distributed mod~$1$, yet $h(\Omega(n))$ is uniformly distributed mod~$1$ due to condition \ref{itm_ud_of_Hardy_functions_along_Omega_Cesaro_ii_b} in Theorem~\ref{thm_ud_of_Hardy_functions_along_Omega_Cesaro}.

Here is another corollary that follows immediately from Theorem \ref{thm_ud_of_Hardy_functions_along_Omega_Cesaro}.

\begin{corollary}
\label{cor_ud_of_Hardy_functions_along_Omega_Cesaro_1}
Let $c>0$. The sequence $\Omega(n)^c$ is uniformly distributed mod~$1$ if and only if $c\in \big(\frac{1}{2},\infty\big)\setminus\N$. The same applies to the sequences $\omega(n)^c$, $\Omega(q_n)^c$, $s_q(n)^c$, and $s_q(p_n)^c$.
\end{corollary}

The surprising conclusion that we can draw from
\cref{thm_ud_of_Hardy_functions_along_Omega_Cesaro} is that if $\vartheta(n)$ satisfies \eqref{gaussian_condition} for some $L\in\mathscr{L}$ then there are many functions $h$ from a Hardy field such that $(h(n))_{n\in\N}$ is uniformly distributed mod~$1$, but $(h(\vartheta(n)))_{n\in\N}$ is not. However, it follows from part~\ref{itm_main_3} of \cref{thm_main} that if one switches from \Cesaro{} averages to averages weighted by $L(N)$ then uniform distribution mod~$1$ along $n$ and along $\vartheta(n)$ become equivalent. 

\begin{theorem}
\label{thm_ud_of_Hardy_functions_along_Omega_loglog}
Let $h$ be a function from a Hardy field with polynomial growth. Then the following are equivalent:
\begin{enumerate}[label=(\roman{enumi}),ref=(\roman{enumi}),leftmargin=*]
\item
\label{ud_of_Hardy_functions_along_Omega_loglog_i}
$\lim_{x\to\infty} \frac{|h(x)-p(x)|}{\log x}=\infty$ for all $p(x)\in \Q[x]$;
\item
\label{ud_of_Hardy_functions_along_Omega_loglog_ii}
$h(n)$ is uniformly distributed mod~$1$;
\item
\label{ud_of_Hardy_functions_along_Omega_loglog_iii}
$h(p_n)$ is uniformly distributed mod~$1$, where $p_n$ denotes the $n$-th prime;
\item
\label{ud_of_Hardy_functions_along_Omega_loglog_iv}
$h(\Omega(n))$ is uniformly distributed mod~$1$ with respect to double-logarithmic averages.
The same applies to the sequences $h(\omega(n))$ and $h(\Omega(q_n))$.
\item
\label{ud_of_Hardy_functions_along_sq_log_v}
$h(s_q(n))$ is uniformly distributed mod~$1$ with respect to logarithmic averages.
The same applies to the sequence $h(s_q(p_n))$.
\end{enumerate}
\end{theorem}

The equivalence between \ref{ud_of_Hardy_functions_along_Omega_loglog_i}, \ref{ud_of_Hardy_functions_along_Omega_loglog_ii}, and \ref{ud_of_Hardy_functions_along_Omega_loglog_iii} in \cref{thm_ud_of_Hardy_functions_along_Omega_loglog} is the content of \cite[Theorem~1.6]{BKS19}. The equivalence between \ref{ud_of_Hardy_functions_along_Omega_loglog_ii} and \ref{ud_of_Hardy_functions_along_Omega_loglog_iv} follows from \eqref{eqn_double-log-averages_for_Omega}.
The equivalence between \ref{ud_of_Hardy_functions_along_Omega_loglog_ii} and \ref{ud_of_Hardy_functions_along_sq_log_v} follows from \eqref{eqn_log-averages_for_sumofdigits}.

It is worth mentioning that we don't know whether it is possible to replace the double-logarithmic averages in part~\ref{ud_of_Hardy_functions_along_Omega_loglog_iv} with logarithmic averages. (However, we think that it is unlikely to be true.)

\subsection*{Applications to ergodic theory}

For the purposes of this paper, a \emph{measure preserving system} will refer to a triple $(X,\mu,T)$ where $X$ is a compact metric space, $T\colon X\to X$ is a continuous map, and $\mu$ is a Borel probability measure on $X$ that is preserved under the transformation $T$, meaning that $\mu(T^{-1}A)=\mu(A)$ holds for all Borel sets $A\subset X$.

Given a point $x\in X$, the sequence $(T^nx)_{n\in\N}$ is called the \emph{orbit} of $x$ under $T$.
Let $C(X)$ denote the space of all (complex-valued) continuous functions on $X$.
The system  $(X,\mu,T)$ is \emph{ergodic} if the orbit of $\mu$-almost every point is uniformly distributed in $X$ with respect to $\mu$, i.e., for any $f \in C(X)$ we have
\[
\lim_{N \to \infty} \frac{1}{N} \sum_{n = 1}^N f(T^{n}x) =\int_X f \ d\mu,\qquad \text{for}~\mu\text{-a.e.}~x\in X.
\]
We call $(X,\mu,T)$ \emph{uniquely ergodic} if the orbit of every point is uniformly distributed in $X$ with respect to $\mu$, that is, for any $f \in C(X)$,
\[
\lim_{N \to \infty} \frac{1}{N} \sum_{n = 1}^N f(T^{n}x) =\int_X f \ d\mu,\qquad\forall x\in X.
\]
Finally, we say $(X,\mu,T)$ is \emph{non-atomic} if the measure $\mu$ is non-atomic.

The following theorem is our dynamical application of \cref{thm_main}. It provides new insights into the behavior of orbits of the form $(T^{\vartheta(n)}x)_{n\in\N}$ in measure-preserving systems, establishing refined equidistribution properties and clarifying the distinctions between different variants of the ergodic theorem along arithmetic functions. 

\begin{theorem}
\label{thm_ergodic}
Assume $\vartheta\colon \N\to\N$ satisfies \eqref{gaussian_condition} for some $L\in\mathscr{L}$.
\begin{enumerate}
\item\label{itm_thm_erg_1}
For any uniquely ergodic measure preserving system $(X,\mu,T)$ and any $f \in C(X)$ we have
    \begin{equation*}
        \lim_{N \to \infty} \frac{1}{N} \sum_{n = 1}^N f(T^{\vartheta(n)}x) =\int_X f \ d\mu,\qquad\forall x\in X.
    \end{equation*}
\item\label{itm_thm_erg_2}
$\vartheta(n)$ has the strong sweeping out property, i.e., for any non-atomic measure preserving system $(X,\mu,T)$ there exists a residual set of Borel sets $B$ such that
\begin{align*}
\limsup_{N \to \infty} \frac{1}{N} \sum_{n = 1}^N 1_B(T^{\vartheta(n)}x) =1,\qquad \text{for}~\mu\text{-a.e.}~x\in X,
\\
\liminf_{N \to \infty} \frac{1}{N} \sum_{n = 1}^N 1_B(T^{\vartheta(n)}x) =0,\qquad \text{for}~\mu\text{-a.e.}~x\in X.
\end{align*}
\item\label{itm_thm_erg_3}
\label{item:floor} If $L(N)=\lfloor W(N)\rfloor$ for some function $W$ from a Hardy field satisfying
\[
\lim_{N\to\infty}\frac{\log(W(N))}{\log(N)}=0,
\]
then for any ergodic measure preserving system $(X,\mu,T)$ and any $f\in L^{\infty}(X,\mu)$ we have
\begin{align*}
\lim_{N \to \infty} \mathbb{E}^{W}_{n\leq N}\,  f(T^{\vartheta(n)}x) =\int_X f \ d\mu,\qquad \text{for}~\mu\text{-a.e.}~x\in X.
\end{align*}
\end{enumerate}
\end{theorem}

When $\vartheta(n)=\Omega(n)$ then part~\ref{itm_thm_erg_1} of \cref{thm_ergodic} was proved in \cite[Theorem~A]{BR22}, part~\ref{itm_thm_erg_2} was shown in \cite{LOYD23}, and part~\ref{itm_thm_erg_3} appeared in \cite[Theorem 1.2]{LM25}.
However, for other choices of $\vartheta(n)$, such as $\Omega(q_n)$, $s_q(n)$, or $s_q(p_n)$, \cref{thm_ergodic} provides new results.

\subsection*{Structure of the paper}

The paper is organized as follows.
The proof of our main technical result, \cref{thm_main}, is split across three sections. In \cref{sec_2} we prove the first part (formula \eqref{eq_main_cesaro_scale}). The proof relies on quantitative estimates for binomial coefficients and elementary results regarding equivalent methods of summation.

In \cref{sec_3} we provide a proof of the second part of \cref{thm_main} (formula \eqref{eq_main_log_scale}).
The principal idea is to show that \eqref{eq_main_cesaro_scale} implies  \eqref{eq_main_log_scale}, and the main ingredient in this derivation is \cref{lem:boshernitzan}, which is a result from an unpublished preprint of Michael Boshernitzan.

In \cref{sec_4}, we first prove that
    \begin{equation}
        \E_{n\leq N}\E_{k\leq n}^{\bin}f(k)=\E_{n\leq N}^{\bin}\E_{k\leq n}f(k) = \E_{n\leq \lfloor{N/2\rfloor}}f(n)+o_{N\to\oo}(1),
    \end{equation}
which is the content of \cref{thm:BC_of_C}.
This result is then used to prove the third and final part of \cref{thm_main} (formula \eqref{eq_main_loglog_scale}).

In \cref{sec_5}, we give conditions for convergence of a sequence with respect to $\E_{n\leq N}^{2\bin}$ averages and use these results to derive Theorem~\ref{thm_ud_of_Hardy_functions_along_Omega_Cesaro} from \cref{thm_main}.

Finally, in \cref{sec_6} we provide a proof of \cref{thm_ergodic}.

\paragraph{Acknowledgments.}We thank Tristán Radić for suggestion that condition \eqref{gaussian_condition} applies to the sequences $s_q(n)$ and $s_q(p_n)$.

\section{Proof of formula \eqref{eq_main_cesaro_scale}}
\label{sec_2}

The goal of this section is to prove the first part of \cref{thm_main}.
For the convenience of the reader, we state this part separately as a theorem.

\begin{theorem}\label{thm:binomial_weights}
Let $W\in\mathscr{W}^*$, $L\in\mathscr{L}$, and assume $\vartheta\colon \N\to\N$ satisfies \eqref{gaussian_condition}. Then uniformly over all $f\colon \N\rightarrow \C$ with $\norm{f}_{\oo}\leq 1$,
    \begin{equation*}
        \E_{n\leq N}f(\vartheta(n)) = \E_{n\leq L(N)}^{2\bin}f(n)+o_{N\to\oo}(1).
    \end{equation*}
\end{theorem}

The main idea behind the proof of \cref{thm:binomial_weights} is to first show that the Gausssian condition \eqref{gaussian_condition} implies
\[
\E_{n\leq N}f(\vartheta(n))\approx \sum_{n\in\N} g(n,L(N),\sqrt{L(N)}) f(n),
\]
and then use the fact that the values of the normalized binomial coefficients $\frac{1}{2^{M+1}}\binom{M}{\floor{m/2}}$ form a close approximation to the gaussian curve with mean $M$ and standard deviation $\sqrt{M}$, which ultimately gives
\[
\sum_{n\in\N} g(n,L(N),\sqrt{L(N)}) f(n) \approx
\sum_{n\in\N} \frac{1}{2^{L(N)+1}}\binom{L(N)}{\floor{n/2}} f(n) = \E_{n\leq L(N)}^{2\bin}f(n).
\]
The details rely on elementary yet technical computations, beginning with \cref{lem:changing_weights}, which characterizes when weighted sums yield equivalent methods of summation.

Let $(\alpha_{n,N})_{n, N\in  \N}$ and $(\beta_{n,N})_{n, N\in\N}$ be nonnegative doubly indexed sequences satisfying
\[
\lim_{N\to\oo}\sum_{n\in \N}\beta_{n,N}=\lim_{N\to\oo}\sum_{n\in \N}\alpha_{n,N}=1.
\]
We seek conditions ensuring that the averages weighted by $(\alpha_{n,N})$
and those weighted by $(\beta_{n,N})$ agree asymptotically, meaning that
\begin{equation}\label{eq:changing_weights}
\sum_{n\in\N} \alpha_{n,N}\, f(n)
=
\sum_{n\in\N} \beta_{n,N}\, f(n)
+ o_{N\to\infty}(1).
\end{equation}
Rewriting equation (\ref{eq:changing_weights}) and using the triangle inequality, we have
\begin{equation}\label{eq:changing_weights_2}
\lim_{N\to\oo}\left|\sum_{n\in \N}\alpha_{n,N}f(n)-\sum_{n\in \N}\beta_{n,N}f(n)\right|\leq 
\|{f}\|_{\oo}\cdot\lim_{N\to\oo}\sum_{n\in \N}|\alpha_{n,N}-\beta_{n,N}|,
\end{equation}
and so it suffices to show that $\lim_{N\to\oo}\sum_{n\in \N}|\alpha_{n,N}-\beta_{n,N}|=0$. To this end, we have the following lemma.

\begin{lemma}\label{lem:changing_weights}
    Suppose that $(\alpha_{n,N})_{n,N\in \N}$ and $(\beta_{n,N})_{n,N\in \N}$ are nonnegative doubly indexed sequences, and $(I_N)_{N\in \N}$ is a sequence of intervals such that 
    \begin{equation}\label{eq:condition_in_changing_weights}
     \lim_{N\to\oo}\sum_{n\in \N}\beta_{n,N}=\lim_{N\to\oo}\sum_{n\in \N}\alpha_{n,N}= \lim_{N\to\oo}\sum_{n\in I_N}\alpha_{n,N}=1.
    \end{equation}
    Assume that there is a function $E\colon\N\rightarrow \R$ which tends to $0$ such that $|1-\frac{\beta_{n,N}}{\alpha_{n,N}}|\leq E(N)$ for all $N\in \N$ and $n\in I_N$.
     Then uniformly over all $f:\mathbb{N}\rightarrow \C$ with $\norm{f}_{\oo}\leq 1$, 
    \begin{equation}
\sum_{n\in \N}\alpha_{n,N}f(n)=\sum_{n\in \N}\beta_{n,N}f(n)+o_{N\to\oo}(1).
    \end{equation}
\end{lemma}

\begin{proof}
Observe that $\lim_{N\to\oo}\sum_{n\in I_N}\beta_n=1$, since
\begin{align*}
  & \lim_{N\to\oo} \sum_{n\in I_N}\beta_n = \lim_{N\to\oo} \sum_{n\in I_N}(\beta_n-\alpha_n) +  \lim_{N\to\oo} \sum_{n\in I_N}\alpha_n\\
   =& \lim_{N\to\oo} \sum_{n\in I_N}\alpha_n\left(\frac{\beta_n}{\alpha_n}-1\right) + 1 = \lim_{N\to\oo} E(N)+1 =1.
\end{align*}
Then 
\begin{align}
    \sum_{n\in \N}|a_{n,N}-\beta_{n,N}| = \sum_{n\in I_N}|a_{n,N}-\beta_{n,N}|+o_{N\to\oo}(1) = \sum_{n\in I_N}\alpha_{n,N}\cdot \left|1-\frac{\beta_{n,N}}{\alpha_{n,N}}\right|+o_{N\to\oo}(1).
\end{align}

Further,

\begin{align*}
 \lim_{N\to\oo}\sum_{n\in I_N}\alpha_{n,N}\cdot \left|1-\frac{\beta_{n,N}}{\alpha_{n,N}}\right|
    &\leq \lim_{N\to\oo}\sum_{n\in I_N}\alpha_{n,N}\cdot E(N)
    \\
    &=\lim_{N\to\oo}E(N)\cdot\sum_{n\in I_N}a_{n,N} 
    \\
    &= 0\cdot 1=0.
\end{align*}
This means that $\lim_{N\to\oo}\sum_{n\in \N}|\alpha_{n,N}-\beta_{n,N}|=0$ as desired.

\end{proof}

\begin{remark}
    By an almost identical argument it can be shown that the condition $|1-\frac{\beta_{n,N}}{\alpha_{n,N}}|\leq E(N)$ in the lemma above can be replaced by the assumption that $|1-\frac{\beta_{n,N}}{\alpha_{n,N}}|\leq E(n)$, so long as $\lim_{N\to\oo}\alpha_{n,N} = 0$ for each fixed $n\in \N$.
\end{remark}

\begin{proof}[Proof of \cref{thm:binomial_weights}]
Recall that $g(x,\mu,\sigma)$ denotes the probability density function of the Gaussian normal distribution. For $n,N\in \N$, define
    \begin{itemize}
        \item $\alpha_{n,N} = \frac{1}{N}\cdot|\{1\leq m\leq N:\vartheta (m)=n\}|$,
        \item $\beta_{n,N} = g(n,L(N),\sqrt{L(N)})$,
        \item $\gamma_{n,N} = g(2\floor{n/2},L(N),\sqrt{L(N)})$,
        \item $\delta_{n,N} = \frac{1}{2^{L(N)+1}}\binom{L(N)}{\floor{n/2}}$.
    \end{itemize}
   We will show that the values 
   \begin{equation*}
       \E_{n\leq N}f(\vartheta(n)),\ \sum_{n\in \N}\alpha_{n,N}f(n),\ \sum_{n\in \N}\beta_{n,N}f(n),\ \sum_{n\in \N}\gamma_{n,N}f(n),\ \sum_{n\in \N}\delta_{n,N}f(n),\ \E_{n\leq L(N)}^{2\bin}f(n)
   \end{equation*}
    are each equal up to a $o_{N\to\oo}(1)$ term, uniformly over all $f\colon \N\rightarrow \C$ with $\norm{f}_{\oo}\leq 1$. First we can note the equalities $ \E_{n\leq N}f(\vartheta(n))=\sum_{n\in \N}\alpha_{n,N}f(n)$ and $\E_{n\leq L(N)}^{2\bin}f(n) =\ \sum_{n\in \N}\delta_{n,N}f(n)$ hold by definition. Next, we have $\sum_{n\in \N}\alpha_{n,N}f(n)=\sum_{n\in \N}\beta_{n,N}f(n)+o_{N\to\oo}(1)$ by equations (\ref{gaussian_condition}) and (\ref{eq:changing_weights_2}).

    The last two equalities will follow from Lemma \ref{lem:changing_weights}. Consider $|1-\frac{\beta_{n,N}}{\gamma_{n,N}}|$. When $n$ is even, we have $\beta_{n,N} = \gamma_{n,N}$ and so $|1-\frac{\beta_{n,N}}{\gamma_{n,N}}|=0$. When $n$ is odd, we have $\gamma_{n,N} = \beta_{n-1,N}$ and so
    \begin{equation}
        \frac{\beta_{n,N}}{\gamma_{n,N}}  = \frac{\frac{1}{\sqrt{L(N)}\cdot \sqrt{2\pi}} \, e^{-\frac{(n-L(N))^2}{2L(N)}}}{\frac{1}{\sqrt{L(N)}\cdot  \sqrt{2\pi}} \, e^{-\frac{((n-1)-L(N))^2}{2L(N)}}} = e^{\frac{2L(N)-2n+1}{2L(N)}} = e^{1-\frac{n}{L(N)}+\frac{1}{2L(N)}}.
    \end{equation}
    
    Next, we will approximate a sum of the form $\sum_{n=A}^B\beta_{n,N}$ with the corresponding integral $\int_A^Bg(x,L(N),\sqrt{L(N)})dx$. However, we know that 
    \[
    \int_A^Bg(x,L(N),\sqrt{L(N)})dx = \int_A^B \frac{1}{\sqrt{L(N)}\cdot \sqrt{2\pi}} \, e^{-\frac{(x-L(N))^2}{2L(N)}}dx = \int_{\frac{A-L(N)}{\sqrt{2L(N)}}}^{\frac{B-L(N)}{\sqrt{2L(N)}}} \frac{1}{\sqrt{\pi}} \, e^{-x^2}dx.
    \]
    From this it follows that we can put $I_N = [L(N) - (L(N))^{3/5},L(N)+(L(N))^{3/5}]$ so that $\sum_{n\in I_N}\beta_{n,N} \to 1$ as $N\to\oo$. But for $n\in I_N$ we have 
    $$
    |1-e^{1-\frac{n}{L(N)}+\frac{1}{2L(N)}}|\leq1- e^{1-\frac{L(N)-(L(N))^{3/5}}{L(N)}+\frac{1}{2L(N)}} = 1-e^{\frac{-1}{L(N)^{2/5}}+\frac{1}{2L(N)}}\to 0 \text{ as }N\to\oo.
    $$
    We can conclude that the hypothesis of Lemma \ref{lem:changing_weights} is satisfied and so $\sum_{n\in \N}\beta_{n,N}f(n) = \sum_{n\in \N}\gamma_{n,N}f(n)+o_{N\to\oo}(1)$. For the last equality, we refer to a fact about the asymptotics of binomial coefficients, whose proof can be found in \cite[Section 5.4]{Spencer14}. Namely, there is a function $E\colon \N\rightarrow \R$  with $\lim_{N\to\oo}E(N)=0$ such that for $|N-2n| = O(N^{2/3})$,
    \[
    \left|1-\frac{\sqrt{\frac{2}{\pi N }}\cdot 2^N\cdot e^{\frac{(N-2n)^2}{2N}}}{\binom{N}{n}}\right|\leq E(N).
    \]
Seeing as how $\frac{\sqrt{\frac{2}{\pi L(N) }}\cdot 2^{L(N)}\cdot e^{\frac{(L(N)-2n)^2}{2L(N)}}}{\binom{L(N)}{n}} = \frac{\frac{1}{\sqrt{2\pi L(N) }}\cdot e^{\frac{(L(N)-2n)^2}{2L(N)}}}{\frac{1}{2^{L(N)+1}}\binom{L(N)}{n}} = \frac{\gamma_{2n,N}}{\delta_{2n,N}}$, it follows that $|1-\frac{\gamma_{n,N}}{\delta_{n,N}}|\leq E(L(N))$ for all $n$ in an interval of the form $[L(N)-O(L(N)^{2/3}),L(N)+O(L(N)^{2/3})]$.  By Lemma \ref{lem:changing_weights}, we get $\sum_{n\in \N}\gamma_{n,N}f(n) = \sum_{n\in \N}\delta_{n,N}f(n) +o_{N\to\oo}(1)$, completing the proof.
\end{proof}

\section{Proof of formula \eqref{eq_main_log_scale}}
\label{sec_3}

In this section, we will provide some background on weighted averages in order to derive equation (\ref{eq_main_log_scale}) from equation (\ref{eq_main_cesaro_scale}). We will begin with some preliminary facts, the first of which is the Stolz-Ces\`aro Theorem

\begin{theorem}[Stolz-Ces\`aro]\label{thm:stolz_cesaro}
Let $A\colon \N\rightarrow \C$ and $B\colon\N\rightarrow (0,\oo)$ be functions such that $B$ is strictly increasing with $\lim_{N\to\oo}B(n) = \oo$. Let $\ell\in \C$. 
\begin{equation}\label{eq:stolz_cesaro}
    \text{ If } \lim_{N\to\oo}\frac{\Delta A(N)}{\Delta B(N)} = \ell \text{ then } \lim_{N\to\oo}\frac{ A(N)}{ B(N)} = \ell.
\end{equation}
\end{theorem}
Next is a standard lemma, say from \cite[Theorem 3.2.7]{Boos}.
\begin{lemma}\label{lem:regularity}
    Let $W\in \mathscr{W}$ and let $f:\N\rightarrow \C$. Suppose that $\lim_{n\to\oo}f(n) = \ell$. Then $\lim_{N\to\oo}\E^W_{n\leq N}f(n) = \ell$.
\end{lemma}
\begin{proof}
    Apply Theorem \ref{thm:stolz_cesaro} with $A(N) = \sum_{n=1}^N\Delta W(n)f(n)$ and $B(N) = W(N)$, so that $\frac{\Delta A(N)}{\Delta B(N)} = \frac{\Delta W(N) f(N)}{\Delta W(N)} = f(N)$.
\end{proof}

When $W$ grows fast enough, we have a converse to this statement.

\begin{lemma}\label{lem:exp_weights}
    Let $W\in \mathscr{W}$ satisfy $\lim_{N\to\oo}\frac{\Delta W(N)}{W(N)}>0$, let $f:\N\rightarrow \C$, and let $\ell\in \C$. Suppose that $\lim_{N\to\oo}\E_{n\leq N}^Wf(n) = \ell$. Then $\lim_{N\to\oo}f(N)=\ell$.
\end{lemma}
\begin{proof}
    Note that 
    \begin{equation*}
   \E_{n\leq N}^Wf(n) =\frac{W(N-1)}{W(N)}\E_{n\leq N-1}^Wf(n)+\frac{\Delta W(N)}{W(N)}f(N),
\end{equation*}
and so 
\begin{equation*}
       f(N) = \left(\frac{\Delta W(N)}{W(N)}\right)^{-1}\left(\E_{n\leq N}^Wf(n)-\frac{W(N-1)}{W(N)}\E_{n\leq N-1}^Wf(n)\right).
\end{equation*}
Observe that $\frac{W(N-1)}{W(N)} = 1-\frac{\Delta W(N)}{W(N)}$, which after taking limits gives that 
\[
\lim_{N\to\oo}f(N) = \frac{\ell-(1-\lim_{N\to\oo}\frac{\Delta W(N)}{W(N)})\cdot \ell}{\lim_{N\to\oo}\frac{\Delta W(N)}{W(N)}}=\ell.
\]
\end{proof}

The next lemma can be attributed to Michael Boshernitzan, who has a variant for Hardy functions in an unpublished preprint \cite{Boshernitzan_preprint}. The proof we provide here is adapted from Boshernitzan's proof.

\begin{lemma}\label{lem:boshernitzan}
   Let $W\in \mathscr{W^*}$ and suppose that $\lim_{N\to\oo}\frac{\log W(N)}{\log(N)}=0$. Then uniformly over all $f:\N\rightarrow \C$ with $\norm{f}_{\oo}\leq 1$,
   \begin{equation}\label{eq:boshernitzan}
        \E_{n\leq N}^W(\E_{k\leq n}f(k)) = \E_{n\leq N}^Wf(n)+o_{N\to\oo}(1).
    \end{equation}
\end{lemma}

\begin{proof}

Recall the summation by parts formula\footnote{This formula for summation by parts holds because of our convention that $\Delta x_n = x_n-x_{n-1}$. A different convention for $\Delta x_n$ would give a different summation by parts formula.}, which says that for sequences $(x_n),(y_n)$,
    \begin{align}\label{eq:summation_by_parts}
        \sum_{n=1}^N\Delta x_n \cdot y_n = x_{N}y_{N}-\sum_{n=1}^{N-1}x_n\cdot \Delta y_{n+1}.
    \end{align}

  Let $f\colon \N\rightarrow \C$ satisfy $\norm{f}_{\oo}\leq 1$. 
  Now we will take the expression for $\E_{n\leq N}^Wf(n)$ and manipulate it into the form $\E_{n\leq N}^W\E_{k\leq n}f(k)+o_{N\to\oo}(1)$. Put $F(n) = \sum_{k=1}^nf(k)$ so that $\Delta F(n) = f(n)$ and $\frac{F(N)}{N}=\E_{n\leq N}f(n)$. Apply summation by parts to obtain
    \begin{align*}
        &\E^W_{n\leq N}f(n) =\frac{1}{W(N)}\sum_{n=1}^N\Delta W(n)\cdot f(n) =\frac{1}{W(N)}\sum_{n=1}^N\Delta W(n)\cdot \Delta F(n)\\
        =& \frac{1}{W(N)}\left(\Delta W(N)\cdot F(N)- \sum_{n=1}^{N-1}F(n)\cdot \Delta^2 W(n+1) \right)\\
        =& \frac{N\cdot \Delta W(N)}{W(N)}\cdot \frac{F(N)}{N}-\frac{1}{W(N)}\sum_{n=1}^{N-1}\Delta W(n)\cdot \frac{F(n)}{n}\cdot \frac{n\cdot \Delta^2W(n+1)}{\Delta W(n)},
    \end{align*}
    which means that 
    \begin{equation}\label{eq:boshernitizan_lem_eq}
        \E^W_{n\leq N}f(n)=\frac{N \Delta W(N)}{W(N)} \E_{n\leq N}f(n)-\frac{W(N-1)}{W(N)} \E_{n\leq N-1}^W\left(\frac{n\cdot \Delta^2W(n+1)}{\Delta W(n)}\E_{m\leq n}f(m)\right).
    \end{equation}
All that is left is to show the following claims:
\begin{enumerate}[label = (\arabic*)]
    \item $\lim_{N\to\oo}\frac{W(N-1)}{W(N)}=1$.
    \item $\lim_{N\to\oo}\frac{N\cdot \Delta^2W(N+1)}{\Delta W(N)}=-1$.
    \item $\lim_{N\to\oo}\frac{N\cdot \Delta W(N)}{W(N)}=0$.
    \item For any bounded function $g\colon \N\rightarrow \C$, $\E_{n\leq N}^Wg(n) = \E_{n\leq N-1}^Wg(n)+o_{N\to\oo}(1)$.
\end{enumerate}
From these claims, equation (\ref{eq:boshernitizan_lem_eq}) becomes
\begin{align*}
    \E^W_{n\leq N}f(n) =&o_{N\to\oo}(1)- (1+o_{N\to\oo}(1))\cdot \E_{n\leq N-1}^W\left(\E_{m\leq n}f(m)\cdot (-1+o_{n\to\oo}(1))\right)\\
    =& \E_{n\leq N-1}^W\left(\E_{m\leq n}f(m)\right)+o_{N\to\oo}(1)\\
    =& \E_{n\leq N}^W\left(\E_{m\leq n}f(m)\right)+o_{N\to\oo}(1).
\end{align*}
as desired.

To prove the above claims, we will make use of Theorem \ref{thm:stolz_cesaro} along with the fact that $\lim_{N\to\oo}\frac{N\cdot \Delta^2 W(N)}{\Delta W(N)}$ exists since $W\in \mathscr{W}^*$. First, we know that $\lim_{N\to\oo}\frac{\log W(N)}{\log N}=0$ and so $\lim_{N\to\oo}\frac{\log W(N)}{ N}=0$. Then,
\begin{equation*}
    0=\lim_{N\to\oo}\frac{\log W(N)}{N} =  \lim_{N\to\oo}\frac{\Delta \log W(N)}{\Delta (N)} = \lim_{N\to\oo}\log\left(\frac{W(N)}{W(N-1)}\right).
\end{equation*}
Therefore $\lim_{N\to\oo}\frac{W(N)}{W(N-1)} = 1$ and we have proven (1). Next, we have
\begin{equation}
    1=\lim_{N\to\oo}\frac{W(N+1)}{W(N)} =  \lim_{N\to\oo}\frac{\Delta W(N+1)}{\Delta W(N)} =\lim_{N\to\oo}\frac{\Delta^2 W(N+1)}{\Delta^2 W(N)} 
\end{equation}
and so to prove (2) and (3), it suffices to show that 
\[
\lim_{N\to\oo}\frac{ N\cdot \Delta W(N+1)}{W(N)}=0
\qquad\text{and}\qquad\lim_{N\to\oo}\frac{N\cdot \Delta^2 W(N+1)}{\Delta W(N)}=-1.
\]
To this end, we will use the approximation $\log(1+x)\approx x$ as $x\to0$ and apply Theorem~\ref{thm:stolz_cesaro} several times to the limit $\lim_{N\to\oo}\frac{\log W(N)}{\log N} = 0$:
\begin{align}
   0=&\lim_{N\to\oo}\frac{\log W(N+1)}{\log (N+1)} =  \lim_{N\to\oo}\frac{\Delta \log W(N+1)}{\Delta \log (N+1)} = \lim_{N\to\oo}\frac{ \log \frac{W(N+1)}{W(N)}}{\log \frac{N+1}{N}} \\
   =&\lim_{N\to\oo}\frac{ \log (1+\frac{\Delta W(N+1)}{W(N)})}{\log (1+\frac{1}{N})}  = \lim_{N\to\oo}\frac{\frac{ \Delta W(N+1)}{W(N)}}{1/N}= \lim_{N\to\oo}\frac{ N\cdot \Delta W(N+1)}{W(N)}\\
   =&  \lim_{N\to\oo}\frac{ \Delta(N\cdot \Delta W(N+1))}{\Delta W(N)} = \lim_{N\to\oo}\frac{ N\cdot \Delta W(N+1)-(N-1)\Delta W(N)}{\Delta W(N)}\\
   =&\lim_{N\to\oo}\frac{N\cdot \Delta^2 W(N+1)}{\Delta W(N)}+1.\label{eq:smoothish_weights_limit}
\end{align}
This shows that $\lim_{N\to\oo}\frac{ N\cdot \Delta W(N+1)}{W(N)}=0$ and $\lim_{N\to\oo}\frac{N\cdot \Delta^2 W(N+1)}{\Delta W(N)}=-1$. 

Claim (4) follows from the fact that 
\begin{align*}
   \E_{n\leq N-1}^Wg(n)=& \frac{1}{W(N-1)}\sum_{n=1}^{N-1}\Delta W(n)g(n) \\
   =&\frac{W(N)}{W(N-1)}\left(\frac{1}{W(N)}\sum_{n=1}^{N}\Delta W(n)g(n) -\frac{\Delta W(N)}{W(N)}g(N)\right)\\
   =&(1+o_{N\to\oo}(1))\cdot\left(\E_{n\leq N}^Wg(n)+o_{N\to\oo}(1) \right) = \E_{n\leq N}^Wg(n)+o_{N\to\oo}(1).
\end{align*}
This completes the proof.
\end{proof}

By applying $\E_{n\leq N}^W$ to both sides of equation (\ref{eq_main_cesaro_scale}) and invoking Lemma \ref{lem:boshernitzan}, we obtain equation (\ref{eq_main_log_scale}) as an immediate corollary:

\begin{corollary}
\label{cor_formula1.9}
  Let $W\in\mathscr{W}^*$, $L\in\mathscr{L}$, and assume $\vartheta\colon \N\to\N$ satisfies \eqref{gaussian_condition}. If $\lim_{N\to\oo}\frac{\log(W(N))}{\log(N)}=0$ then uniformly over all $f\colon \N\to\C$ with $\|f\|_\infty\leq 1$,
\begin{equation*}
\E_{n\leq N}^W f(\vartheta(n)) =\E_{n\leq N}^W\E_{k\leq L(N)}^{2\bin}f(k)+o_{N\to\oo}(1).
\end{equation*}
\end{corollary}

\section{Proof of formula \eqref{eq_main_loglog_scale}}
\label{sec_4}

The goal of this section is to prove formula \eqref{eq_main_loglog_scale}, which is the third and final part of \cref{thm_main}.
Let us state this result as a standalone theorem.

\begin{theorem}
\label{thm_1.2.3}
Let $W\in\mathscr{W}^*$, $L\in \mathscr{L}$, and assume $\vartheta\colon \N\to\N$ satisfies \eqref{gaussian_condition}.
If $\lim_{N\to\oo}\frac{\log (W\circ L)(N)}{\log(N)}=0$ then uniformly over all $f\colon \N\to\C$ with $\|f\|_\infty\leq 1$,
\begin{equation*}
\E_{n\leq N}^{W\circ L} f(\vartheta(n))  = \E^{W}_{n\leq L(N)}f(n)+o_{N\to\oo}(1).
\end{equation*}
\end{theorem}

We will derive \cref{thm_1.2.3} from \cref{cor_formula1.9}; one of the key components in this derivation is the following theorem.

\begin{theorem}\label{thm:BC_of_C}
Uniformly over all $f\colon \N\rightarrow \C$ with $\norm{f}_{\oo}\leq 1$,
    \begin{equation}
    \label{eqn_BC_of_C_1}
        \E_{n\leq N}\E_{k\leq n}^{\bin}f(k)=\E_{n\leq N}^{\bin}\E_{k\leq n}f(k) = \E_{n\leq \lfloor{N/2\rfloor}}f(n)+o_{N\to\oo}(1).
    \end{equation}
\end{theorem}

From \cref{thm:BC_of_C}, we obtain the following immediate corollary.

\begin{corollary}\label{cor:C_of_2BC}
    Uniformly over all $f\colon \N\rightarrow \C$ with $\norm{f}_{\oo}\leq 1$,
    \begin{equation}    \label{eqn_BC_of_C_2}
        \E_{n\leq N}\E_{k\leq n}^{2\bin}f(k)= \E_{n\leq N}f(n)+o_{N\to\oo}(1).
    \end{equation}
\end{corollary}

\begin{proof}
Using the definition of $\E^{2\bin}$ and invoking \cref{thm:BC_of_C}, we get
\begin{align*}
\E_{n\leq N}\E_{k\leq n}^{2\bin}f(k)
&=
\frac{1}{2}
\E_{n\leq N}\E_{k\leq n}^{\bin}f(2k)
+
\frac{1}{2}
\E_{n\leq N}\E_{k\leq n}^{\bin}f(2k+1)
\\
&=
\frac{1}{2}
\E_{n\leq \lfloor{N/2\rfloor}} f(2n)
+
\frac{1}{2}
\E_{n\leq \lfloor{N/2\rfloor}}f(2n+1)+o_{N\to\oo}(1)
\\
&=
\E_{n\leq N} f(n)
+o_{N\to\oo}(1),
\end{align*}
as desired.
\end{proof}

It remains to prove \cref{thm:BC_of_C}. 
The leftmost equality in \eqref{eqn_BC_of_C_1} follows from the following fact.
\begin{lemma}
For any function $f:\N\rightarrow \C$ and any $N\in \N$ we have
\begin{equation}
    \E_{n\leq N}\E^{\bin}_{k\leq n}f(k) = \E_{n\leq N}^{\bin}\E_{k\leq n}f(k).
\end{equation}
\end{lemma}

The proof of this lemma can be found in \cite{Boos} by combining Proposition 3.4.4(e) with Example 3.4.7 and Definition 3.4.8. It remains to prove the rightmost equality in Theorem \ref{thm:BC_of_C}. Our strategy will be to show that $\E_{n\leq N}^{\bin}\E_{k\leq n}f(k)$ is close to an average of $\E_{k\leq n}f(k)$ for $n$ near $\floor{N/2}$, and that each term of this form is very close to $\E_{k\leq \floor{N/2}}f(k)$. The idea for this proof strategy, albeit phrased slightly differently, can be found in \cite{Gajser16}.

\begin{lemma}\label{lem:cesaro_along_subsequences}
    Let $f:\N\rightarrow \C$ be bounded and $N,M \in \N$.
    Then
    \begin{equation*}
        |\E_{n\leq N}f(n)-\E_{n\leq M}f(n)|\leq \left(\frac{2|N-M|}{\min\{M,N\}+1}\right)\cdot \|{f}\|_{\oo}.
    \end{equation*}
\end{lemma}

\begin{proof}
    
Without loss of generality, assume that $N\geq M$. 
    \begin{align*}
        |\E_{n\leq N}f(n)-\E_{n\leq M}f(n)| = &\left|\frac{1}{N+1}\sum_{n=0}^{N}f(n)-\frac{1}{M+1}\sum_{n=0}^{M}f(n)\right|\\
        =& \left|\sum_{n=0}^{M}\left(\frac{f(n)}{N+1}-\frac{f(n)}{M+1}\right)+\frac{1}{N+1}\sum_{n=M+1}^{N}f(n)\right|\\
        =& \left|\frac{M-N}{N+1}\cdot \frac{1}{M+1}\sum_{n=0}^{M}f(n)+\frac{1}{N+1}\sum_{n=M+1}^{N}f(n)\right|\\
        \leq&  \frac{|M-N|}{N+1}\cdot \frac{1}{M+1}\sum_{n=0}^{M}|f(n)|+\frac{1}{N+1}\sum_{n=M+1}^{N}|f(n)|
        \\=&  \frac{N-M}{N+1}\cdot \E_{n\leq M}|f(n)|+\frac{1}{N+1}\sum_{n=M+1}^{N}|f(n)|
        \\\leq &  \frac{N-M}{N+1}\cdot \|f\|_{\oo}+\frac{N-M}{N+1}\|f\|_{\oo}\\
        =&  \frac{2(N-M)}{N+1}\|f\|_{\oo}.
    \end{align*}

\end{proof}

\begin{corollary}
Let $(A_N)_{N\in \N}$ and $(B_N)_{N\in \N}$ be positive integer-valued sequences with $\lim_{N\to\oo}B_N = \oo$ and  $\lim_{N\to\oo}\frac{A_N}{B_N}= 1$. Then uniformly over all $f:\N\rightarrow \C$ with $\norm{f}_{\oo}\leq 1$, 
\[
\E_{n\leq A_N}f(n) = \E_{n\leq B_N}f(n)+o_{N\to\oo}(1).
\]
\end{corollary}

Now for the proof of Theorem~\ref{thm:BC_of_C}.

\begin{proof}[Proof of Theorem~\ref{thm:BC_of_C}]
Let $X_1,\dots, X_N$ be i.i.d. random variables with $\P(X_1=0)=\P(X_i=1)=1/2$, so that $\sum_{i=1}^NX_i\sim \text{Bin}(N,1/2)$. The central limit theorem states that the sequence $\frac{\sum_{i=1}^NX_i}{\sqrt{N}}$ converges in distribution to $\mathcal{N}(0,1)$ as $N\to\oo$. So for any $A<B\in \R$, 
\[
\lim_{N\to\oo}\P\left(A<\frac{\sum_{i=1}^NX_i}{\sqrt{N}}<B\right) = \P(A<\mathcal{N}(0,1)<B).
\]
But 
\[
\lim_{N\to\oo}\P\left(A<\frac{\sum_{i=1}^NX_i}{\sqrt{N}}<B\right) =\P\left(A\sqrt{N}<{\sum_{i=1}^NX_i}<B\sqrt{N}\right)=\frac{1}{2^N}\sum_{n=N/2+A\sqrt{N}}^{N/2+B\sqrt{N}}\binom{N}{n}.
\]
Taking $A\to-\oo$ and $B\to\oo$ we have $\P(A<\mathcal{N}(0,1)<B)\to 1$. It follows that whenever $D$ is a function which tends to $\oo$ faster than $\sqrt{N}$, we have $\frac{1}{2^N}\sum_{n=N/2-D(N)}^{N/2+D(N)}\binom{N}{n}\to 1$ as $N\to\oo$.

For each $N\in \N$, put $I_{N} = [\floor{N/2-N^{2/3}},\floor{N/2+N^{2/3}}]$. Then $\lim_{N\to\oo}2^{-N}\cdot\sum_{n\in I_N}\binom{N}{n}=1$,
and so
    \begin{align*} 
    \left|\E_{n\leq N}^{\bin}\E_{k\leq n}f(k)-\frac{1}{2^N}\sum_{n\in I_N}\binom{N}{n}\E_{k\leq n}f(k)\right| < \|f\|_{\oo}\cdot o_{N\to\oo}(1).
    \end{align*}

Now we will apply Lemma \ref{lem:cesaro_along_subsequences}. For any $n\in I_N$ we have that 
\begin{align*}
    |\E_{k\leq \lfloor{N/2\rfloor}}f(k)-\E_{k\leq n}f(k)| \leq &\left(\frac{2|\floor{N/2}-n|+1}{\min\{\floor{N/2},n\}+1}\right)\cdot \|{f}\|_{\oo}\\
    \leq &\left(\frac{2\lceil{N^{2/3}\rceil}+1}{\lfloor{N/2 - N^{2/3}\rfloor}+1}\right)\cdot \|{f}\|_{\oo}\\
    \leq& \left(\frac{2}{\lfloor{N^{1/3}/2-1\rfloor}}+o_{N\to\oo}(1)\right)\cdot \|{f}\|_{\oo} = o_{N\to\oo}(1). 
\end{align*}
    So $|\E_{k\leq \lfloor{N/2\rfloor}}f(k)-\E_{k\leq n}f(k)| $ goes to $0$ as $N\to\oo$ uniformly for $n\in I_N$. Hence,
    \begin{align*}
        \frac{1}{2^N}\sum_{n\in I_N}\binom{N}{n}\E_{k\leq n}f(k) =& \frac{1}{2^N}\sum_{n\in I_N}\binom{N}{n}\E_{k\leq \floor{N/2}}f(k)+o_{N\to\oo}(1)\\
        =& \E_{k\leq \floor{N/2}}f(k)\cdot \left(\frac{1}{2^N}\sum_{n\in I_N}\binom{N}{n}\right)+o_{N\to\oo}(1).
    \end{align*}
In total we have that
\begin{align*}
    \E_{n\leq N}^{\bin}\E_{k\leq n}f(k)
    =&\frac{1}{2^N}\sum_{n\in I_N}\binom{N}{n}\E_{k\leq n}f(k)+o_{N\to\oo}(1)\\ =& \E_{k\leq \floor{N/2}}f(k)\cdot \left(1+o_{N\to\oo}(1)\right)+o_{N\to\oo}(1)\\
    =&\E_{k\leq \floor{N/2}}f(k)+o_{N\to\oo}(1),
\end{align*}
which concludes the proof.
\end{proof}

The final ingredient in the proof of \cref{thm_1.2.3}
is a discrete counterpart of the change-of-variables formula for integrals. We introduce this identity next and include a proof for completeness.
Recall the standard formula, which states that  
\begin{equation}
\label{eqn_int_changeofvariable_formula}
\int_{1}^{N} (W\circ s)'(x)\, f(s(x))\, dx
    \;=\;
\int_{s(1)}^{s(N)} W'(y)\, f(y)\, dy.
\end{equation}
The following proposition is a discrete variant of \eqref{eqn_int_changeofvariable_formula}.

\begin{proposition}
\label{cor:change_of_variables_no_s_hat}
   Let $W\in \mathscr{W}$, $s\in\mathscr{L}$, and suppose that $\lim_{N\to\oo}\frac{\Delta W(N)}{W(N)}=0$.
Then uniformly over all $f:\N\rightarrow \C$ with $\norm{f}_{\oo}\leq 1$,
   \begin{equation*}
    \E^{W\circ s}_{n\leq N}(f(s(n)))= \E_{k\leq s(N)}^{W}(f(k))+ o_{N\to\oo}(1).
    \end{equation*}
\end{proposition}

For the proof of \cref{cor:change_of_variables_no_s_hat}, we use the next lemma.

\begin{lemma}\label{lem:change_of_variables}
   Let $s\in \mathscr{L}$ and define $\hat{s}(k)= \max\{n:s(n)\leq k\}$. Let $W\in \mathscr{W}$ and suppose that $\lim_{N\to\oo}\frac{\Delta(W\circ \hat{s})(N)}{(W\circ \hat{s})(N)}=0$. Then uniformly over all $f:\N\rightarrow \C$ with $\norm{f}_{\oo}\leq 1$,
   \begin{equation*}
    \E^W_{n\leq N}(f(s(n)))= \E_{k\leq s(N)}^{W\circ \hat{s}}(f(k))+ o_{N\to\oo}(1).
    \end{equation*}
\end{lemma}

The statement of \cref{cor:change_of_variables_no_s_hat} follows simply by rephrasing Lemma~\ref{lem:change_of_variables} to remove any reference to $\hat{s}$ by replacing $W$ with $W\circ s$.   
It remains to prove Lemma~\ref{lem:change_of_variables}.

\begin{proof}[Proof of Lemma~\ref{lem:change_of_variables}]
First we will note that $\hat{s}$ is a right inverse for $s$, so that $s(\hat{s}(k))=k$ for all $k\in \N$. When $N\in \N$ is equal to $\hat{s}(M)$ for some $M\in \N$, we have that $\hat{s}(s(N)) = \hat{s}(s(\hat{s}(M))) = \hat{s}(M) = N$. So for $N$ contained in the image of $\hat{s}$, we can calculate that
\begin{align*}
       \E^W_{n\leq N}(f(s(n)))=& \frac{1}{W(N)}\sum_{n=1}^N\Delta W(n) f(s(n)) =\frac{1}{W(N)}\sum_{k=1}^{s(N)}\sum_{n\leq N: s(n)=k}\Delta W(n)f(k).
\end{align*}
But for each $k\leq s(N)$, $\{n\leq N:s(n)=k\}$ is the interval $\{ \hat{s}(k-1)+1,\dots, \hat{s}(k)\}$, so
    \[
    \sum_{n\leq N: s(n)=k}\Delta W(n) = \sum_{n=\hat{s}(k-1)+1}^{\hat{s}(k)}\Delta W(n)=W(\hat{s}(k))-W(\hat{s}(k-1)) = \Delta (W\circ \hat{s})(k).
    \]
    Additionally, writing $W(N) = W(\hat{s}(s(N)) = (W\circ \hat{s})(s(N))$, we have
    \[
   \frac{1}{W(N)}\sum_{k=1}^{s(N)}\sum_{n\leq N: s(n)=k}\Delta W(n)f(k) = \frac{1}{(W\circ \hat{s})(s(N))}\sum_{k=1}^{s(N)}\Delta (W\circ \hat{s})(k)\cdot f(k) = \E_{k\leq s(N)}^{W\circ \hat{s}}f(k).
    \]
    It total, we have shown that $\E^W_{n\leq N}(f(s(n))) = \E_{k\leq s(N)}^{W\circ \hat{s}}f(k)$ whenever $N$ belongs to the image of $\hat{s}$.
    
Now suppose that $N$ does not belong to the image of $\hat{s}$, so that we have  $\hat{s}(s(N)-1)<N<\hat{s}(s(N))$. Then $\{n\leq N:s(n)=s(N)\}$ is the interval $\{ \hat{s}(s(N)-1)+1,\dots, N\}$, so
    \[
    \sum_{n\leq N: s(n)=s(N)}\Delta W(n) = \sum_{n=\hat{s}(s(N)-1)+1}^{N}\Delta W(n)=W(N)-W(\hat{s}(s(N)-1)) 
    \]
By the argument from the first half of the proof, we have
\begin{align*}
&\E_{n\leq N}^Wf(s(n))=\frac{1}{W(N)}\sum_{n=1}^{\hat{s}(s(N))}\Delta W(n)f(s(n))-\frac{1}{W(N)}\sum_{n=N+1}^{\hat{s}(s(N))}\Delta W(n)f(s(n))\\
=& \frac{(W\circ \hat{s})(s(N))}{W(N)}\cdot \frac{1}{(W\circ \hat{s})(s(N))}\sum_{k=1}^{s(N)}\Delta (W\circ \hat{s})(k) f(k)-\frac{1}{W(N)}\sum_{n=N+1}^{\hat{s}(s(N))}\Delta W(n)f(s(n))\\
=& \frac{(W\circ \hat{s})(s(N))}{W(N)}\cdot \E_{k\leq s(N)}^{W\circ \hat{s}}f(k)-\frac{1}{W(N)}\sum_{n=N+1}^{\hat{s}(s(N))}\Delta W(n)f(s(n)).
\end{align*}
Now we claim that $\lim_{N\to\oo} \frac{\Delta(W\circ \hat{s})(s(N))}{W(N)}=0$ since
\[
\lim_{N\to\oo} \frac{\Delta(W\circ \hat{s})(s(N))}{W(N)} \leq \lim_{N\to\oo} \frac{\Delta(W\circ \hat{s})(s(N))}{W(\hat{s}(s(N)))}=\lim_{N\to\oo} \frac{\Delta(W\circ \hat{s})(s(N))}{W(\hat{s}(s(N)))} =0
\]
by the fact that $W$ is eventually increasing, $\hat{s}(s(N))\geq N$, and our assumption that $\lim_{N\to\oo}\frac{\Delta (W\circ \hat{s})(N)}{(W\circ \hat{s})(N)}=0$. Noting that
\begin{align*}
\left|\frac{1}{W(N)}\sum_{n=N+1}^{\hat{s}(s(N))}\Delta W(n)f(s(n))\right|\leq & \frac{\norm{f}_{\oo}}{W(N)}\sum_{n=N+1}^{\hat{s}(s(N))}\Delta W(n)\\
\leq &\frac{ \norm{f}_{\oo}}{W(N)}\cdot(W(\hat{s}(s(N)))-W(N))\\
\leq &\norm{f}_{\oo}\cdot \frac{ \Delta(W\circ \hat{s})(N)}{W(N)}=o_{N\to\oo}(1),
\end{align*}
we have
\[
\E_{n\leq N}^Wf(s(n)) = (1+o_{N\to\oo}(1)) \cdot \E_{k\leq s(N)}^{W\circ \hat{s}}f(k)-o_{N\to\oo}(1) = \E_{k\leq s(N)}^{W\circ \hat{s}}f(k)+o_{N\to\oo}(1),
\]
which means that we are done.

\end{proof}

\begin{remark}\label{remark:s_hat}
    When $s(n) = \floor{q^{-1}(n)}$ for some increasing function $q\colon \R\rightarrow \R$ with $q(\N) \subset \N$ and $\Delta q^{-1}(n)\leq 1$ for all $n\in \N$, we have $\hat{s}(n) = q(n)$.
\end{remark}

\begin{example}
    Take $W(N) = N$ and $s(N) = \floor{\sqrt{N}}$ so that $\hat{s}(N) = N^2$. Then for any bounded function $f:\N\rightarrow \C$ we have that 
\begin{align*}
    \E_{1\leq n\leq N}f({\lfloor\sqrt{n}\rfloor}) = \frac{1}{N}\sum_{n=1}^Nf(\lfloor\sqrt{n}\rfloor )= &\frac{1}{(\lfloor\sqrt{N}\rfloor)^2}\sum_{n=1}^{\lfloor \sqrt{N}\rfloor}(2n+1)f(n)+o_{N\to\oo}(1)\\=&  \E_{1\leq n\leq \lfloor \sqrt{N}\rfloor}^{V}f(n)+o_{N\to\oo}(1) \text{ for } V(N) = N^2.
\end{align*}
\end{example}

\begin{proof}[Proof of \cref{thm_1.2.3}]
   \cref{cor_formula1.9} gives us that $\E_{n\leq N}^{W \circ L}f(n)=\E_{n\leq N}^{W \circ L}\E_{k\leq L(n)}^{2\bin}f(k)+o_{N\to\oo}(1)$. Now we will use \cref{cor:change_of_variables_no_s_hat} to obtain
   \[
   \E_{n\leq N}^{W \circ L}\E_{k\leq L(n)}^{2\bin}f(k) =\E_{n\leq L(N)}^W\E_{k\leq n}^{2\bin}f(k)+o_{N\to\oo}(1).
   \]
   We may apply \cref{cor:change_of_variables_no_s_hat} because $W\in \mathscr{W}^*$ and $\lim_{N\to\oo}\frac{\log(W(N))}{N}=0$, so that we may compare equations (\ref{eqn_smoothish_weights}) and (\ref{eq:smoothish_weights_limit}) to see that $\lim_{N\to\oo}\frac{\Delta W(N)}{W(N)} = 0$.

  Now apply \cref{lem:boshernitzan}, \cref{cor:C_of_2BC}, and \cref{lem:boshernitzan} again, so that we have
    \begin{align*}
        \E_{n\leq L(N)}^W\E_{k\leq n}^{2\bin}f(k)
        =& \E_{n\leq L(N)}^W\E_{k\leq n}\E_{m\leq k}^{2\bin}f(m)+o_{N\to\oo}(1)\\
        =& \E_{n\leq L(N)}^W\E_{k\leq n}f(k)+o_{N\to\oo}(1)\\
        =& \E_{n\leq L(N)}^Wf(n)+o_{N\to\oo}(1),
    \end{align*}
completing the proof.
\end{proof}

\section{Proof of Theorem~\ref{thm_ud_of_Hardy_functions_along_Omega_Cesaro}}
\label{sec_5}

The goal of this section is to prove Theorem~\ref{thm_ud_of_Hardy_functions_along_Omega_Cesaro} (or rather, an equivalent form which we formulate now).
Note that part~\ref{itm_main_1} of \cref{thm_main} tells us that $(h(\vartheta(n)))_{n\in \N}$ is uniformly distributed mod~$1$ with respect to regular \Cesaro{} averages if and only if $(h(n))_{n\in \N}$ is uniformly distributed mod~$1$ with respect to $\E^{2\bin}$ averages.
This allows us to state Theorem~\ref{thm_ud_of_Hardy_functions_along_Omega_Cesaro} in the following equivalent way.

\begin{theorem}
\label{thm_ud_of_Hardy_functions_along_2bin_means}
Let $h$ be a Hardy field function with polynomial growth.
The following are equivalent:
\begin{enumerate}[label=(\roman{enumi}),ref=(\roman{enumi}),leftmargin=*]
\item
The sequence $(h(n))_{n\in \N}$ is uniformly distributed mod~$1$ with respect to $\E^{2\bin}$ averages.
\item 
One of the following two (mutually exclusive) conditions is satisfied:

\begin{enumerate}[label=(\alph{enumii}),ref=(\alph{enumii}),leftmargin=*]
\item
\label{itm_ud_of_Hardy_functions_along_2bin_ii_a}
$\lim_{x\to\infty} \frac{|h(x)-p(x)|}{x \log x}=\infty$ for all $p(x)\in \Q[x]$;
\item
\label{itm_ud_of_Hardy_functions_along_2bin_ii_b}
$\lim_{x\to\oo}\frac{|h(x)-p(x)|}{\sqrt{x}}=\oo$ for each $p(x)\in \Q[x]$ and there exists $q(x)\in \Q[x]$ such that $\lim_{x\to\oo}\frac{|h(x)-q(x)|}{x}<\oo$.
\end{enumerate}
\end{enumerate}
\end{theorem}
The rest of this section is devoted to the proof of \cref{thm_ud_of_Hardy_functions_along_2bin_means}. 
We begin by recalling that Hardy functions are totally ordered by asymptotic growth rate. Thus, we can prove \cref{thm_ud_of_Hardy_functions_along_2bin_means} by considering cases. Let $h$ be a function of polynomial growth belonging to a Hardy field. Exactly one of the following statements is true.
\begin{enumerate}[label = (\arabic*)]
    \item There exists $q(x)\in \Q[x]$ such that $\lim_{x\to\oo}\frac{|h(x)-q(x)|}{\sqrt{x}}<\oo$,
    \item  $\lim_{x\to\oo}\frac{|h(x)-p(x)|}{\sqrt{x}}=\oo$ for all $p(x)\in \Q[x]$ and there exists $q(x)\in \Q[x]$ such that $\lim_{x\to\oo}\frac{|h(x)-q(x)|}{x}<\oo$,
    \item $\lim_{x\to\oo}\frac{|h(x)-p(x)|}{x}=\oo$ for all $p(x)\in \Q[x]$ and there exists $q(x)\in \Q[x]$ such that $\lim_{x\to\oo}\frac{|h(x)-q(x)|}{x\log(x)}=0$,
    \item $\lim_{x\to\oo}\frac{|h(x)-p(x)|}{x\log(x)}>0$ for all $p(x)\in \Q[x]$ and there exists $q\in \Q[x]$ such that $0<\lim_{x\to\oo}\frac{|h(x)-q(x)|}{x\log(x)}<\oo$,
    \item $\lim_{x\to\oo}\frac{|h(x)-p(x)|}{x\log(x)}=\oo$ for all $p(x)\in\Q[x]$.
\end{enumerate}

It is evident that conditions (2) and (5) above are identical to conditions (b) and (a) in \cref{thm_ud_of_Hardy_functions_along_2bin_means}, respectively. We will show that if either of conditions (2) or (5) hold then $(h(n))_{n\in \N}$ is uniformly distributed mod $1$ with respect to $\E^{2\bin}$ averages, and we will also show that if any of conditions (1), (3), or (4) hold then $(h(n))_{n\in\N}$ is not uniformly distributed mod $1$ with respect to $\E^{2\bin}$ averages.

 Given a function $f\colon \N\rightarrow \C$ and an interval of natural numbers $[a,b]$, we will find it convenient to use the notation
\begin{equation*}
\E_{n\in [a,b]}f(n)= \frac{1}{b-a}\sum_{n=a}^bf(n).
\end{equation*}

First we will consider the cases where one of conditions (2) or (5) holds. It suffices to prove the following theorems.

\begin{theorem}\label{thm:s_implies_w_implies_bin}
       Let $f:\N\rightarrow \C$ be bounded, let $\ell\in \C$, and let $W(x) = e^{\sqrt{x}}$. 
       Consider the following statements.
       \begin{enumerate}[label = (\roman*)]
           \item There exists a function $V$ belonging to a Hardy field satisfying $\lim_{N\to\oo}\frac{\log(W(N))}{\log(V(N))}=0$ and
           \[
           \lim_{N\to\oo}\E_{n\in [N-s(N),N]}f(n)=\ell
           \]
           for each $s\colon \N\rightarrow \N$ with $\lim_{N\to\oo}s(N)\cdot \Delta \log(V(N))= \oo$,
          \item  $\lim_{N\to\oo}\E^W_{n\leq N}f(n)=\ell$,
          \item $\lim_{N\to\oo}\E^{2\bin}_{n\leq N}f(n)=\ell$.
       \end{enumerate}
       Then (i)$\implies $(ii)$\implies$(iii).
\end{theorem}

\begin{theorem}\label{thm:w_ud}
Let $h$ be a function with polynomial growth which belongs to a Hardy field and let $W(x)=e^{\sqrt{x}}$. Suppose that $\lim_{x\to\oo}\sqrt{x} |h'(x)-p(x)|=\infty$ for all $p(x)\in \Q[x]$ and there is some $q(x)\in \Q[x]$ such that $\lim_{x\to\infty}|h'(x)-q(x)|<\infty$.
Then 
\begin{equation}
\lim_{N\to\oo}\E^W_{n\leq N}e^{2\pi i k h(n)}=0
\end{equation}
for each $k\in \Z\setminus \{0\}$.
\end{theorem}

\begin{theorem}\label{thm:s_ud}
Let $h$ be a function with polynomial growth which belongs to a Hardy field. Suppose that
\begin{equation} \label{eq:s_ud}
\lim_{x\to\infty} \frac{|h'(x)-p(x)|}{\log x}=\infty \text{ for all }p(x)\in \Q[x].
\end{equation}
Then there exists a function $V\in \mathscr{W}^*$ which belongs to a Hardy field and satisfies 
\[
\lim_{N\to\oo}\frac{\sqrt{N}}{\log(V(N))}=0
\]
such that 
\begin{equation}
\lim_{N\to\oo}\E_{n\in [N-s(N),N]}e^{2\pi i k h(n)}=0
\end{equation}
for all $k\in \Z\setminus \{0\}$ and all $s\colon\N\rightarrow \N$ with $\lim_{N\to\oo}s(N)\cdot \Delta \log(V(N))=\oo$ and $s(N)\leq N-1$ for all $N\in \N$.
\end{theorem}

The first implication of Theorem \ref{thm:s_implies_w_implies_bin} follows from this next theorem.
\begin{theorem}[{\cite[Theorem C]{reilly26}}]\label{thm:mikey_thm}
   Suppose that $V\in \mathscr{W}^*$ belongs to a Hardy field and satisfies $\lim_{N\to\oo}\frac{\log(V(N))}{\log (N)}=\oo$ and $\lim_{N\to\oo}\frac{\log(V(N))}{ N}=0$.  Let $f\colon \N\rightarrow \C$ be bounded, and let $\ell\in \C$. The following statements are equivalent:
   \begin{enumerate}[label = (\arabic*)]
       \item $\lim_{N\to\oo}\E_{n\leq N}^Wf(n) = \ell$ for each 
       $W\in \mathscr{W}^*$ which belongs to the same Hardy field as $V$ and satisfies $\lim_{N\to\oo}\frac{\log(W(N))}{\log(V(N))}=0$,
       \item $\lim_{N\to\oo} \E_{n\in [N-s(N),N]}f(n) = \ell$ for all nondecreasing functions $s\colon\N\rightarrow\N$ which satisfy $\lim_{N\to\oo}s(N) \cdot \Delta \log(V(N))=\oo$ and $s(N)\leq N-1$ for all $N\in \N$.
   \end{enumerate}
\end{theorem}

The second implication in Theorem \ref{thm:s_implies_w_implies_bin} follows from \cite[Theorem 3.2.8]{Boos} and Lemma \ref{lem:bin_stronger_than_e_sqrt} below.
\begin{theorem}[{\cite[Theorem 3.2.8]{Boos}}]\label{thm:Boos}
    Let $W\in \mathscr{W}$ and let $(\alpha_{n,N})_{n,N\in \N}$ be a nonnegative doubly indexed sequence such that $\lim_{N\to\oo}\sum_{n\in \N}\alpha_{n,N}=1$. Then the following are equivalent:
    \begin{itemize}
        \item For each function $f\colon \N\rightarrow \C$ and each $\ell\in \C$, if $\lim_{N\to\oo}\E_{n\leq N}^Wf(n) = \ell$ then $\lim_{N\to\oo}\sum_{n\in \N}\alpha_{n,N}f(n) = \ell$.
        \item  $\sup_{N\to\oo} \sum_{n\in \N}W(n)\left|\frac{\alpha_{n,N}}{\Delta W(n)}-\frac{\alpha_{n+1,N}}{\Delta W(n+1)}\right|<\oo$, and for each $N\in \N$, $\lim_{n\to\oo}\frac{\alpha_{n,N}}{\Delta W(n)} = 0$.
    \end{itemize}
\end{theorem}
\begin{lemma}\label{lem:bin_stronger_than_e_sqrt}
   Let $f:\N\rightarrow \C$ be any function, let $\ell\in \C$ and let $W(x) = e^{\sqrt{x}}$. Suppose that $\lim_{N\to\oo}\E_{n\leq N}^Wf(n)=\ell$. Then $\lim_{N\to\oo}\E_{n\leq N}^{2\bin}f(n) = \lim_{N\to\oo}\E_{n\leq N}^{\bin}f(n)  = \ell$.
\end{lemma}

\begin{proof}

We will apply Theorem \ref{thm:Boos} with $\alpha_{n,N} = \frac{1}{2^N}\binom{N}{n}$ and $W(x) = e^{\sqrt{x}}$. It is clear that we have $\lim_{n\to\oo}\frac{\alpha_{n,N}}{\Delta W(n)} = 0$ for each $N$, since $\alpha_{n,N}=0$ for all $n>N$. Next, we will consider 
\begin{align}\label{eq:messy_sum}
\sum_{n\in \N}W(n)\left|\frac{\alpha_{n,N}}{\Delta W(n)}-\frac{\alpha_{n+1,N}}{\Delta W(n+1)}\right| = \frac{1}{2^N}\sum_{n\in \N}\left|\frac{W(n)\binom{N}{n}}{\Delta W(n)}-\frac{W(n)\binom{N}{n+1}}{\Delta W(n+1)}\right| .
\end{align}

Note that $\frac{ W(n)}{ W(n+1)} = e^{\sqrt{n}-\sqrt{n+1}} = 1+O_{N\to\oo}(\Delta \sqrt{n+1}) = 1+O_{N\to\oo}(n^{-1/2})$.
Put $\eta(n)= \frac{W(n)\binom{N}{n}}{2^N\Delta W(n)}$ so that we have
\begin{align*}
   & \frac{1}{2^N}\sum_{n\in \N}\left|\frac{W(n)\binom{N}{n}}{\Delta W(n)}-\frac{W(n)\binom{N}{n+1}}{\Delta W(n+1)}\right| =\sum_{n\in \N}\left|\eta(n)-\eta(n+1)(1+O_{N\to\oo}(n^{-1/2}))\right|   \\
    \leq &\sum_{n\in \N}\left|\eta(n)-\eta(n+1)\right| +\sum_{n\in \N}\eta(n+1)\cdot O_{n\to\oo}(n^{-1/2}).
    \end{align*}
    The second sum is bounded since
    \begin{align*}
    &\sum_{n=1}^{\oo}\eta(n+1)\cdot O_{n\to\oo}(n^{-1/2}) = \sum_{n=1}^{\oo}\frac{1}{2^N}\binom{N}{n+1}\frac{W(n+1)}{\Delta W(n+1)}\cdot O_{n\to\oo} (n^{-1/2})\\
     =& \sum_{n=1}^{\oo}\frac{1}{2^N}\binom{N}{n+1}\cdot O_{n\to\oo}(n^{1/2})\cdot O_{n\to\oo} (n^{-1/2}) = O_{N\to\oo}(1).
    \end{align*}
      To bound the other sum above, note that the ratio $\frac{\eta(n+1)}{\eta(n)}\sim \frac{N-n}{n+1}$ uniformly in $N$. Since $\frac{N-n}{n+1}$ is decreasing in $n$ this shows that, when $n$ is large, $\eta(n)$ increases to its maximum and then decreases. So $\sum_{n\in \N}\left|\eta(n)-\eta(n+1)\right|
 \leq  2\cdot \sup_{n\in \N}\eta(n)$. We can bound $\sup_{n\leq N}\eta(n)$ by noting that $\binom{N}{n}\leq \frac{2^N}{\sqrt{\pi N}}$ for all $n$ and $\frac{W(n)}{\Delta W(n)}\leq \frac{W(N)}{\Delta W(N)} = 2\sqrt{N}\cdot (1+o_{N\to\oo}(1))$ for all $n\leq N$. In particular, $\max_{n\leq N}\eta(n) = O_{N\to\oo}(1)$, and so we can conclude that $\E_{n\leq N}^{\bin}f(n)=\ell$ by Theorem \ref{thm:Boos}.

 In order to prove that $\E_{n\leq N}^{2\bin}f(n)=\ell$, we can take $\alpha_{n,N}  = \frac{1}{2^{N+1}}\binom{N}{\floor{n/2}}$ and perform a similar calculation as above. 
 
\end{proof}

 Now that we have shown Theorem~\ref{thm:s_implies_w_implies_bin}, we will consider Theorem \ref{thm:w_ud}. Suppose that $\lim_{x\to\oo}\sqrt{x} |h'(x)-p(x)|=\infty$ for all $p(x)\in \Q[x]$ and there is a $q(x)\in \Q[x]$ such that $\lim_{x\to\infty}|h'(x)-q(x)|<\infty$. By L'H\^opital's rule, this means that $\lim_{x\to\oo}\frac{|h(x)-p(x)|}{\sqrt{x}}=\infty$ for all $p(x)\in \Q[x]$ and there is $Q(x)\in \Q[x]$ such that $\lim_{x\to\infty}\frac{|h(x)-Q(x)|}{x}<\infty$.

Let $\alpha,\beta\in\R\setminus \Q$, $k\in \Z\setminus \{0\}$, and let $W(n)=e^{\sqrt{n}}$. Put $r(n) = \beta(h(n)-Q(n))$ so that $r(n)$ grows faster than $\sqrt{n}$ and slower than $n$. Consider 
$$
\lim_{N\to\oo}\E_{n\leq N}^We^{2\pi i k \alpha \floor{r(n)} }.
$$
By Lemma~\ref{lem:change_of_variables} we have 
$$
\lim_{N\to\oo}\E_{n\leq N}^We^{2\pi i k \alpha \floor{r(n)} }=\lim_{N\to\oo}\E_{n\leq r(N)}^{W\circ r^{-1}}e^{2\pi i k \alpha n },
$$
but we know that $\lim_{N\to\oo}\E_{n\leq N}^{W\circ r^{-1}}e^{2\pi i k \alpha n }=0$ by Theorem \ref{thm:mikey_thm} and the fact that $\{n\alpha \}_{n\in \N}$ is well distributed mod 1.
So we have reduced Theorem \ref{thm:w_ud} to the following lemma.

\begin{lemma}
    Let $W(n) = e^{\sqrt{n}}$. Let $r\colon \N\rightarrow \N$ and suppose that 
    \begin{equation} 
    \lim_{N\to\oo}\E_{n\leq N}^We^{2\pi i \alpha \floor{\beta r(n)} }=0
    \end{equation}
    for all $\alpha,\beta\in \R\setminus \Q$. Then
    \begin{equation} 
    \lim_{N\to\oo}\E_{n\leq N}^We^{2\pi i k r(n) }=0
    \end{equation}
    for all $k\in \Z\setminus \{0\}$.
\end{lemma}
\begin{proof}

Let $k\in \Z\setminus \{0\}$ and pick any $\epsilon>0$ with $\epsilon\not\in \Q$. Write $r(n) = \epsilon(\epsilon^{-1}r(n)\text{ mod } 1)+\epsilon\floor{\epsilon^{-1}r(n)}$. Since $\epsilon(\epsilon^{-1}r(n)\text{ mod }1)\in [0,\epsilon)$ for all $n$, we have 
    \begin{align*}
    &\limsup_{N\to\oo}\left|\E_{n\leq N}^We^{2\pi i kr(n)}-\E_{n\leq N}^We^{2\pi i k \epsilon \floor{\epsilon^{-1}r(n)} }\right|\\
    =&\limsup_{N\to\oo}\left|\E_{n\leq N}^We^{2\pi i k \epsilon \floor{\epsilon^{-1}r(n)} }e^{2\pi i k \epsilon (\epsilon^{-1}r(n)\text{ mod }1) }-\E_{n\leq N}^We^{2\pi i k \epsilon \floor{\epsilon^{-1}r(n)} }\right|\\    =&\limsup_{N\to\oo}\left|\E_{n\leq N}^We^{2\pi i k \epsilon \floor{\epsilon^{-1}r(n)} }(e^{2\pi i k \epsilon (\epsilon^{-1}r(n)\text{ mod } 1) }-1)\right|\\
=&\limsup_{N\to\oo}\E_{n\leq N}^W|e^{2\pi i k \epsilon (\epsilon^{-1}r(n)\text{ mod } 1) }-1|\\
    \leq& \limsup_{n\to\oo} |e^{2\pi i k \epsilon (\epsilon^{-1}r(n)\text{ mod } 1) }-1|< 2\pi |k|\epsilon \to 0\text{ as }\epsilon\to 0. 
    \end{align*}
    So we are done.
\end{proof}

Before giving a proof of Theorem \ref{thm:s_ud}, recall van der Corput's trick.
\begin{theorem}[van der Corput's trick]\label{thm:vdc_trick}
Let $(x_n)_{n\in \N}$ be a bounded sequence of complex numbers and let $(I_N)_{N\in \N}$ be a sequence of intervals of natural numbers with $|I_N|\to\oo$ as $N\to\oo$. Suppose that for each $j\in \N$, $\E_{n\in I_N}(x_{n+j}\overline{x_n}) \to 0$ as $N\to\oo$. Then $\E_{n\in I_N} x_n \to 0$ as $N\to\oo$.
\end{theorem}

\cref{thm:vdc_trick} is a special case of \cite[Theorem~2.12]{BM16} when $F_N=I_N$, $G=\Z$, and~$H=\C$.

\begin{corollary}\label{cor:vdc_trick}
Let $h\colon \R\rightarrow \R$ be a Hardy function of polynomial growth and let $(I_N)_{N\in \N}$ be a sequence of intervals of natural numbers with $|I_N|\to\oo$ as $N\to\oo$. Suppose that $\E_{n\in I_N}e^{2\pi i h'(n)} \to 0$ as $N\to\oo$. Then $\E_{n\in I_N} e^{2\pi i h(n)} \to 0$ as $N\to\oo$.
\end{corollary}

Now we are ready to give a proof of \cref{thm:s_ud}.
\begin{proof}[Proof of \cref{thm:s_ud}]
By replacing $h(x)$ with $k(h(x)-\int_0^x p(t)dt)$ if necessary, we can assume without loss of generality that $k=1$ and $p(x)=0$. Additionally, assume that $h$ eventually increases to $\oo$.

Let $m\geq 1$ be such that $\lim_{x\to\oo}\frac{h(x)}{x^m}=\oo$ and $\lim_{x\to\oo}\frac{h(x)}{x^{m+1}}<\oo$. 
We proceed by considering cases. 
\begin{itemize}
    \item Case 1: $m=1$,
    \item Case 2: $m=2$
    \item Case 3: $m\geq 3$.
\end{itemize}

We can begin by observing that Case 3 reduces to Case 2. Indeed, if $\lim_{x\to\oo}\frac{h(x)}{x^m}=\oo$ and $\lim_{x\to\oo}\frac{h(x)}{x^{m+1}}<\oo$, then $\lim_{x\to\oo}\frac{h^{(m-2)}(x)}{x^2}=\oo$ and $\lim_{x\to\oo}\frac{h^{(m-2)}(x)}{x^{3}}<\oo$ and so we may apply Case 2 to $h^{(m-2)}$ and invoke Corollary \ref{cor:vdc_trick}.

Now we will turn our attention to Case 1. Suppose that $m=1$. Applying L'H\^opital's rule to equation (\ref{eq:s_ud}) gives us that $\lim_{x\to\oo}\frac{h'(x)}{\log x}=\lim_{x\to\oo}x\cdot h''(x)=\oo$. Let 
\[
V(N) = e^{\int_{n=1}^N{\sqrt{h''(n)}}}
\]
so that $\Delta\log(V(N)) = \sqrt{h''(N)}\cdot (1+o_{N\to\oo}(1))$ and 
\[
\lim_{N\to\oo}\frac{\sqrt{N}}{\log(V(N))} = \lim_{N\to\oo}\frac{\Delta \sqrt{N}}{\Delta \log(V(N)) }=\lim_{N\to\oo}\frac{\frac{1}{2\sqrt{N}}}{\sqrt{h''(N)}} =\lim_{N\to\oo}\frac{1}{2\sqrt{N\cdot h''(N)} }=0.
\]

Let $s\colon \N\rightarrow \N$ be any function with $s(N)\leq N-1$ for all $N\in \N$ and $\lim_{N\to\oo}s(N)\cdot \Delta \log(V(N)) = \oo$ so that $\lim_{N\to\oo}s(N)\sqrt{h''(N)}=\oo$. We will show that 
\begin{equation} \label{eq:goal}
\lim_{N\to\oo}\frac{1}{s(N)}\sum_{n=N-s(N)}^{N}e^{2\pi i  h(n)}=0.
\end{equation}

If $\lim_{N\to\oo}\frac{s(N)}{N}\in (0,1]$ then (\ref{eq:goal}) holds because $(h(n))_{n\in \N}$ is uniformly distributed in the usual sense, so suppose that $s(N) = o_{N\to\oo}(N)$. 

Theorem 2.2 in \cite{GK91} says that if $h$ is a smooth function and $I$ is an interval with $\lambda \leq |h{''}(x)|\leq \alpha \lambda$ for $x\in I$ then
\begin{equation}\label{eq:vdc_2.2}
    \frac{1}{|I|}\left|\sum_{n\in I}e^{2\pi i k h(n)}\right| \leq C\left(\alpha\ \lambda^{1/2}+|I|^{-1}\lambda^{-1/2}\right)
\end{equation}
for some uniform constant $C>0$.

Put $\lambda = h''(N)$ and $\alpha = \frac{h''(N-s(N))}{h''(N)}$ and $I= [N-s(N),N]$ for $N\in \N$. We have that $\lambda\to 0$ as $N\to\oo$ and that $\lim_{N\to\oo}\alpha<\oo$ since $s(N) = o_{N\to\oo}(N)$ and $h''(x)$ tends to $0$ slower than $x^{-1}$. So $\alpha \cdot \lambda^{1/2}\to 0$ as $N\to\oo$.

Additionally, $|I|\cdot \lambda^{1/2} = s(N)\cdot \sqrt{h''(N)}\to \oo$ as $N\to\oo$. So, the right-hand side of equation (\ref{eq:vdc_2.2}) tends to $0$ as $N\to\oo$, so it follows that $\left|\E_{n\in I}e^{2\pi i h(n)} \right|=    \frac{1}{|I|}\left|\sum_{n\in I}e^{2\pi i h(n)}\right| \to 0$ as $N\to\oo$. This completes the proof for Case 1.

Lastly, consider Case 2. In this case, we have $\lim_{x\to\oo}\frac{h(x)}{x^2}=\oo$. Then by L'H\^opital's rule, we also have
\[
\oo=\lim_{x\to\oo}\frac{h(x)}{x^{3/2}} = \lim_{x\to\oo}\frac{h'(x)}{\frac{3}{2}x^{1/2}}  = \lim_{x\to\oo}\frac{h''(x)}{\frac{3}{4}x^{-1/2}} = \lim_{x\to\oo}\frac{h'''(x)}{\frac{-3}{8}x^{-3/2}} =\lim_{x\to\oo}\frac{-8}{3}h'''(x)x^{3/2} .
\]

Let 
\[
V(N) = e^{\int_{n=1}^N{\sqrt[3]{h'''(n)}}}
\]
so that $\Delta\log(V(N)) = \sqrt[3]{h'''(N)}\cdot(1+o_{N\to\oo}(1))$ and 
\[
\lim_{N\to\oo}\frac{\sqrt{N}}{\log(V(N))} = \lim_{N\to\oo}\frac{\Delta \sqrt{N}}{\Delta \log(V(N)) }=\lim_{N\to\oo}\frac{\frac{1}{2\sqrt{N}}}{\sqrt[3]{h'''(N)}} =\lim_{N\to\oo}\frac{1}{2\sqrt[3]{N^{3/2}\cdot h'''(N)} }=0.
\]

Let $s\colon \N\rightarrow \N$ be any function with $s(N)\leq N-1$ for all $N\in \N$ and $\lim_{N\to\oo}s(N)\cdot \Delta \log(V(N)) = \oo$ so that $\sqrt[4]{h'''(N)\cdot s(N)^3}\to \oo$. 
Theorem 2.6 in \cite{GK91} says that if $h$ is a smooth function and $I$ is an interval with $\lambda \leq |h'''(x)|\leq \alpha \lambda$ for $x\in I$ then

\begin{equation}\label{eq:vdc_2.6}
    \frac{1}{|I|}\left|\sum_{n\in I}e^{2\pi i h(n)}\right| \leq C\left(\alpha^{1/3} \lambda^{1/6}+\alpha^{1/4}|I|^{-1/4}+\lambda^{-1/4}|I|^{-3/4}\right)
\end{equation}
for some uniform constant $C>0$. 

Put $\lambda = h^{'''}(N)$ and $\alpha = \frac{h^{(3)}(N-s(N))}{h^{'''}(N)}$ and $I=I(N) = [N-s(N),N]$ for $N\in \N$. We have that $\lim_{N\to\oo}\alpha<\oo$ and so $\alpha^{1/3} \lambda^{1/6}+\alpha^{1/4}|I|^{-1/4}\to 0$ as $N\to\oo$. Also, $\lambda^{1/4}|I|^{3/4} = \sqrt[4]{ h^{(3)}(N)\cdot s(N)^3} \to \oo$ as $N\to\oo$.

The right-hand side of equation (\ref{eq:vdc_2.6}) tends to $0$ as $N$ tends to $\oo$, so it follows that $\left|\E_{n\in I_N}e^{2\pi i h(n)}\right| = \frac{1}{|I|}\left|\sum_{n\in I}e^{2\pi i h(n)}\right| \to 0$ as $N\to\oo$. This concludes the proof of Case 2, and so we are done.
\end{proof}

It remains to show that if any of conditions (1), (3), or (4) hold then $h(n)$ is not uniformly distributed mod $1$ with respect to $\E^{2\bin}$ averages.

Suppose that condition (1) holds. There is a Tauberian theorem in \cite[Theorem 157, p.~221]{H49}\footnote{In \cite{H49} the limit $\lim_{N\to\oo}\E^{\bin}_{n\leq N}$ is referred to as the $(E,1)$ method} which says that if $\lim_{N\to\oo}\E_{n\leq N}^{\bin}\sum_{j=1 }^na_j = \ell$ and $a_n = o(n^{-1/2})$ then $\lim_{n\to\oo}\sum_{j=1}^na_j = \ell$. Taking $a_n = (e^{2\pi i k h(2n)}-e^{2\pi i k h(2n-2)}+e^{2\pi i k h(2n+1)}-e^{2\pi i k h(2n-1)})/2$ gives that $a_n = o(n^{-1/2})$
and $\sum_{j=1}^na_j = (e^{2\pi i kh(2n)}+e^{2\pi i kh(2n+1)}- e^{2\pi i kh(0)}-e^{2\pi i kh(1)})/2$ which does not tend to $0$ as $n\to\oo$. So it follows that 
\[
\lim_{N\to\oo}\E_{n\leq N}^{2\bin}e^{2\pi i k h(n)}=\lim_{N\to\oo}\E_{n\leq N}^{\bin}\sum_{j=1 }^na_j \neq 0.
\]

Next, we can consider the case where condition (3) holds.

\begin{lemma}\label{lem:tangent_line} Let $I = [a,b]$ with $a$ arbitrarily large and let $h$ be a hardy function with $\lim_{x\to\oo}\frac{|h(x)|}{x}=\oo$ and $\lim_{x\to\oo}\frac{h(x)}{x^2}=0$. Let $c\in [a,b]$, and let $L_c(x) = h'(c)(x-c)+h(c)$ be the tangent line to $h$ at $x=c$. Then $|e^{2\pi i h(n)}-e^{2\pi i L_c(n)}|\leq 2\pi |I|^2h''(a)$ for all $n\in I$. Notably, this bound is independent of $c$.
\end{lemma}
\begin{proof}
    Recall the chord inequality, which says that $|e^{i\theta}-e^{i\phi}|\leq |\theta-\phi|$ for all $\theta,\phi\in \R$. So 
\[
|e^{2\pi i h(n)}-e^{2\pi i L_c(n)}|\leq 2\pi |h(n)-L_c(n)| =  2\pi |h(n)-h(c)-h'(c)(n-c)|.
\]
By Taylor's theorem $|h(n)-h(c)-h'(c)(n-c)|\leq (n-c)^2h''(a)/2\leq |I|^2h''(a)$.
\end{proof}

\begin{theorem} Let $h$ be a Hardy function with polynomial growth and suppose that $\lim_{x\to\oo}\frac{|h(x)-p(x)|}{x}=\oo$ for all $p(x)\in \Q[x]$ and there exists $q(x)\in \Q[x]$ such that $\lim_{x\to\oo}\frac{|h(x)-q(x)|}{x\log(x)}=0$. Then $\lim_{N\to\oo}\E_{n\leq N}^{\bin}e^{2\pi i h(n)}\neq 0$. 
\end{theorem}
\begin{proof}

We will show that $\limsup_{N\to\oo}|\E_{n\leq N}^{\bin}e^{2\pi i h(n)}|=1$. Without loss of generality assume that $q(x)=0$ and $h$ increasing to $\oo$.

Let $\epsilon>0$ and pick $C>0$ such that for $ I = [N/2-C\sqrt{N}, N/2+C\sqrt{N}]$ we have $\liminf_{N\to\oo}\frac{1}{2^N}\sum_{n\in I}\binom{N}{n}>1-\epsilon$.

It follows that 
\begin{equation}\label{3}
\left|\frac{1}{2^N}\sum_{n=1}^N\binom{N}{n}e^{2\pi i h(n)}-\frac{1}{2^N}\sum_{n\in I}\binom{N}{n}e^{2\pi i h(n)}\right|<\epsilon.
\end{equation}

Note that $h'$ eventually increases to $\oo$. So, for infinitely many values of $N$ we can find a real number $c\in I$ such that $h'(c)\in \Z$. Pick a large enough $N$ such that there is a value $c\in I$ with $h'(c)\in \Z$. Let $L_c(x) = h'(c)(x-c) + h(c)$ be the tangent line to $h$ at $x=c$. Note that $e^{2\pi i L_c(n)}$ is constant for all integers $n\in I$ and so 
\begin{equation}\label{1}
\left|\frac{1}{2^N}\sum_{n\in I}\binom{N}{n}e^{2\pi i L_c(n)}\right|=\frac{1}{2^N}\sum_{n\in I}\binom{N}{n}>1-\epsilon.
\end{equation}

By Lemma \ref{lem:tangent_line}, $|e^{2\pi i h(n)}-e^{2\pi i L_c(n)}|\leq 2\pi |I|^2h''(N/2-C\sqrt{N})$ for all $n\in I$. In particular, $|I|^2 = 4C^2N$ and by using L'H\^opital's rule on the hypothesis of the theorem, we have that $\lim_{x\to\oo}h''(x)\cdot x =0$. It follows that $\lim_{N\to\oo}|I|^2h''(N/2-C\sqrt{N}) = 0$. Then
\begin{equation}\label{2}
\lim_{N\to\oo}\left|\frac{1}{2^N}\sum_{n\in I}\binom{N}{n}e^{2\pi i h(n)}-\frac{1}{2^N}\sum_{n\in I}\binom{N}{n}e^{2\pi i L_c(n)}\right|=0.
\end{equation}

Combining (\ref{3}), (\ref{1}), and (\ref{2}) we get that $\limsup_{N\to\oo}\left|\E_{n\leq N}^{\bin}e^{2\pi i h(n)}\right|>1-2\epsilon$.
\end{proof}

Lastly, consider the case where condition (4) holds.

\begin{lemma}
\label{lem_dioph_approx_of_f}
Suppose that $h$ is a Hardy function such that $\lim_{x\to\oo}\frac{h(x)}{x\log(x)}>0$ and there exists $0<\lim_{x\to\oo}\frac{h(x)}{x\log(x)}<\oo$. There exist infinitely many $N\in\N$ such that
\[
\|h'(N)\|_{\Z}\leq h''(N),
\]
where $\|x\|_{\Z}$ denotes the distance from $x$ to the closest integer.
\end{lemma}

\begin{proof}
Let $m$ be a large positive integer. $h'$ eventually increases to $\oo$ and so we can pick $N\in\N$ such that $h'(N)$ is smaller than $m$, but $h'(N+1)$ is bigger or equal than $m$.
By using the mean value theorem, we can note that the inequality
\[
|h'(N+1)-h'(N)|= |h''(c)|\leq h''(N) \text{ for some } c\in (N,N+1)
\]
holds for all but at most finitely many $N\in\N$, since $h''$ is eventually decreasing. Since $m$ lies between $h'(N)$ and $h'(N+1)$, it follows that
\[
|m-h'(N)|\leq h''(N).
\]
This completes the proof.
\end{proof}

\begin{lemma}
\label{lem_binomial_averages_of_nlogn}
Suppose that $h$ is a Hardy function such that $\lim_{x\to\oo}\frac{|h(x)-p(x)|}{x\log(x)}>0$ for all $p(x)\in \Q[x]$ and there exists $q\in \Q[x]$ such that $0<\lim_{x\to\oo}\frac{|h(x)-q(x)|}{x\log(x)}<\oo$. Let $k\in \Z\setminus \{0\}$. Then there exists a constant $A\in \R$ such that for every $\epsilon>0$ and infinitely many $N\in\N$,
\begin{equation}\label{eq:nlogn_approx}
\E^{2\bin}_{n \le N} e^{2\pi i kh(n)}= e^{2\pi i kh(N)}\cdot \Bigg(\frac{1}{\sqrt{2\pi}}\int_{-\infty}^\infty e^{-\frac{x^2}{2}+A\pi i x^2} \d x \Bigg)+\Oh(\epsilon).
\end{equation}
\end{lemma}

\begin{proof}
Without loss of generality, assume that $h$ eventually increases to $\oo$, $k=1$, and $q(x)=0$. Let $\epsilon>0$, pick $C\in \N$ arbitrarily large and put
\[
I_N= \Z\cap [N-C\sqrt{N}, N+C\sqrt{N}]\qquad\text{and}\qquad \tilde{I}_N= \Z\cap [-C\sqrt{N},C\sqrt{N}].
\]
If $C$ is chosen sufficiently large we have
\begin{equation}\label{eq:central_limit}
\frac{1}{2^{N+1}}\sum_{n\in I_N}\binom{N}{\big\lfloor \frac{n}{2} \big\rfloor}  \geq 1-\epsilon \text{ for all sufficiently large } N.
\end{equation}
From the proof of Theorem \ref{thm:binomial_weights}, we get
\begin{align}
\E^{2\bin}_{n \le N} e^{2\pi i h(n)}
&=\frac{1}{\sqrt{2\pi N}} \sum_{n\in I_N}  e^{-\frac{\big(\frac{N-n}{\sqrt{N}}\big)^2}{2}} e^{2\pi i h(n)} + \Oh(\epsilon)+o_{N\to\oo}(1)
\\
&=\frac{1}{\sqrt{2\pi N}} \sum_{n\in \tilde{I}_N}  e^{-\frac{\big(\frac{n}{\sqrt{N}}\big)^2}{2}} e^{2\pi i h(N+n)} + \Oh(\epsilon)+o_{N\to\oo}(1).
\end{align}
Using a second-degree Taylor approximation for $h(x)$ at the point $x=N$, we have
\[
h(N+n)= h(N)+ h'(N)\cdot  n + \frac{n^2}{2}\cdot h''(N) + \Oh\Big(n^3\cdot h'''(N)\Big).
\]
Let $A = \lim_{N\to\oo}\frac{h(N)}{N\log(N)} = \lim_{N\to\oo}\frac{h''(N)}{1/N}$ and note that $A\in (0,\oo)$. The $\Oh\Big(n^3\cdot h'''(N)\Big)$ term above tends to $0$ since for $n\in \tilde{I}_N$ we have,
\begin{align*}
|n^3\cdot h'''(N)| \leq |N^{3/2}\cdot h'''(N)| =&\frac{A^{3/2}}{(h''(N))^{3/2}}\cdot (-h'''(N))\cdot (1+o_{N\to\oo}(1))\\=&2A^{3/2}\cdot\left(\frac{1}{(h''(N))^{1/2}}\right)'\cdot (1+o_{N\to\oo}(1)),
\end{align*}
and
\[
\lim_{N\to\oo}\frac{\left(\frac{1}{(h''(N))^{1/2}}\right)'}{1}=\lim_{N\to\oo}\frac{\left(\frac{1}{(h''(N))^{1/2}}\right)}{N}=\sqrt{\lim_{N\to\oo}\frac{1/N^2}{h''(N)}}=\sqrt{\lim_{N\to\oo}\frac{\log(N)}{h(N)}}=0.
\]

Using \cref{lem_dioph_approx_of_f}, we can choose $N$ arbitrarily large such that $\|h'(N)\|_{\Z}\leq h''(N)$ and hence
\[
\max_{n\in\tilde{I}_N}\{ \|h'(N) \cdot n\|_{\Z}\} \leq h''(N)\cdot \sqrt{N}.
\]
But we can also note that 
\[
\lim_{N\to\oo}h''(N)\cdot \sqrt{N} =\lim_{N\to\oo} \frac{h''(N)}{1/\sqrt{N}} = \lim_{N\to\oo} \frac{h'(N)}{2\sqrt{N}}=\lim_{N\to\oo} \frac{h(N)}{4N^{3/2}/3}=0.
\]
 Note that $n^2h''(N) = A\cdot \frac{n^2}{N}\cdot(1+o_{N\to\oo}(1))$.
Then uniformly over all $n\in\tilde{I}_N$ we have
\begin{align*}
e^{2\pi i h(N+n)}&= e^{2\pi i h(N)} \cdot e\Big(\frac{An^2}{2N}\cdot (1+o_{N\to\oo}(1))\Big)+\oh_{N\to\infty}(1)
\\
&=e^{2\pi i h(N)}  \cdot  e^{A\pi i \left(\frac{n}{\sqrt{N}}\right)^2}+\oh_{N\to\infty}(1).
\end{align*}
Combined with the above, we thus have
\begin{align}
\E^{2\bin}_{n \le N}e^{2\pi i h(n)} &=
e^{2\pi i h(N)} \cdot \Bigg(
\frac{1}{\sqrt{2\pi N}} \sum_{n\in \tilde{I}_N}  e^{-\frac{\big(\frac{n}{\sqrt{N}}\big)^2}{2}} e^{A\pi i \big(\frac{n}{\sqrt{N}}\big)^2}\Bigg) + \Oh(\epsilon)+o_{N\to\oo}(1).
\end{align}

Observe that the above sum is a Riemann sum. More precisely, we have
\[
\frac{1}{\sqrt{2\pi N}} \sum_{n\in \tilde{I}_N}  e^{-\frac{\big(\frac{n}{\sqrt{N}}\big)^2}{2}} e^{A\pi i \big(\frac{n}{\sqrt{N}}\big)^2}
=
\frac{1}{\sqrt{2\pi}}\int_{-A}^A e^{-\frac{x^2}{2}+A\pi i x^2} \d x + \oh_{N\to\infty}(1).
\]
Adding back the tails of the integral gives another error in the order of $\Oh(\epsilon)$, and so we have
\[
\E^{2\bin}_{n \le N}e^{2\pi i h(n)} = e^{2\pi i h(N)} \cdot \Bigg(\frac{1}{\sqrt{2\pi}}\int_{-\infty}^\infty e^{-\frac{x^2}{2}+A\pi i x^2} \d x \Bigg)+\Oh(\epsilon)+o_{N\to\oo}(1).
\]
\end{proof}

Recall the classical fact that 
\[
\int_{-\infty}^{\infty} e^{-\alpha x^2}\,dx
= \sqrt{\frac{\pi}{\text{Re}(\alpha)}}
\]  
when $\text{Re}(\alpha)>0$. It readily follows from equation (\ref{eq:nlogn_approx}) that 
\[
\limsup_{N\to\infty}\Big|\E^{2\bin}_{n \le N} e^{2\pi i kh(n)}\Big| >0
\]
when condition (4) holds. This completes the proof of Theorem \ref{thm_ud_of_Hardy_functions_along_2bin_means}.

\section{Proof of Theorem~\ref{thm_ergodic}}
\label{sec_6}

The main ingredient in the proof of  Theorem~\ref{thm_ergodic}, aside from our main result \cref{thm_main}, is the following proposition.

\begin{proposition}
\label{prop_2bin_ergodic_averages}
\begin{enumerate}
\item\label{itm_prop_erg_1}
For any uniquely ergodic measure preserving system $(X,\mu,T)$ and any $f \in C(X)$ we have
    \begin{equation*}
        \lim_{N \to \infty} \E_{n\leq N}^{2\bin} f(T^{n}x) =\int_X f \ d\mu,\qquad\forall x\in X.
    \end{equation*}
\item\label{itm_prop_erg_2}
For any non-atomic measure preserving system $(X,\mu,T)$ there exists a residual set of Borel sets $B$ such that
\begin{align*}
\limsup_{N \to \infty} \E_{n\leq N}^{2\bin} 1_B(T^{n}x) =1,\qquad \text{for}~\mu\text{-a.e.}~x\in X,
\\
\liminf_{N \to \infty}\E_{n\leq N}^{2\bin} 1_B(T^{n}x) =0,\qquad \text{for}~\mu\text{-a.e.}~x\in X.
\end{align*}
\end{enumerate}
\end{proposition}

\begin{proof}
We will give an ``indirect'' proof of \cref{prop_2bin_ergodic_averages} by combining existing results in the literature concerning ergodic averages along $\Omega(n)$ with \cref{thm_main}.
It is worth mentioning that one could also prove \cref{prop_2bin_ergodic_averages} directly without mentioning results concerning $\Omega(n)$.

By \cref{thm_main} we have
\begin{equation}
\label{eqn_dynamical_omega_to_2bin}
\frac{1}{N} \sum_{n = 1}^N f(T^{\Omega(n)}x) =\E_{n\leq \lfloor \log\log N\rfloor}^{2\bin} f(T^{n}x) + \oh_{N\to\infty}(1).
\end{equation}
By \cite[Theorem A]{BR22}, for any uniquely ergodic system $(X,\mu,T)$ and any $f \in C(X)$,
\begin{equation}
\lim_{N\to\infty} \frac{1}{N} \sum_{n = 1}^N f(T^{\Omega(n)}x)=\int_X f \ d\mu,\qquad\forall x\in X.
\end{equation}
Since any subsequence of a convergent sequence is convergent, we also have
\begin{equation}
\lim_{N\to\infty} \frac{1}{\exp(\exp(N))} \sum_{n = 1}^{\exp(\exp(N))} f(T^{\Omega(n)}x)=\int_X f \ d\mu,\qquad\forall x\in X.
\end{equation}
Combined with \eqref{eqn_dynamical_omega_to_2bin}, this proves that
\begin{equation}
\lim_{N\to\infty} \E_{n\leq N}^{2\bin} f(T^{n}x) = \int_X f \ d\mu,\qquad\forall x\in X,
\end{equation}
and part \ref{itm_prop_erg_1} of \cref{prop_2bin_ergodic_averages} follows.

For part \ref{itm_prop_erg_2} we follow a similar strategy. 
By \cite{LOYD23}, we know that for any non-atomic measure preserving system $(X,\mu,T)$ there exists a residual set of Borel sets $B$ such that
\begin{align*}
\limsup_{N \to \infty} \frac{1}{N} \sum_{n = 1}^N 1_B(T^{\Omega(n)}x) =1,\qquad \text{for}~\mu\text{-a.e.}~x\in X,
\\
\liminf_{N \to \infty} \frac{1}{N} \sum_{n = 1}^N 1_B(T^{\Omega(n)}x) =0,\qquad \text{for}~\mu\text{-a.e.}~x\in X.
\end{align*}
By \eqref{eqn_dynamical_omega_to_2bin}, we obtain 
\begin{align*}
\limsup_{N \to \infty} \E_{n\leq \lfloor \log\log N\rfloor}^{2\bin} 1_B(T^{n}x) =1,\qquad \text{for}~\mu\text{-a.e.}~x\in X,
\\
\liminf_{N \to \infty}\E_{n\leq \lfloor \log\log N\rfloor}^{2\bin} 1_B(T^{n}x) =0,\qquad \text{for}~\mu\text{-a.e.}~x\in X,
\end{align*}
and hence part \ref{itm_prop_erg_2} follows.
\end{proof}

\begin{proof}[Proof of \cref{thm_ergodic}]
Assume $\vartheta\colon \N\to\N$ satisfies \eqref{gaussian_condition} for some $L\in\mathscr{L}$.
By \cref{thm_main}, we have
\begin{equation}
\label{eqn_dynamical_vartheta_to_2bin}
\frac{1}{N} \sum_{n = 1}^N f(T^{\vartheta(n)}x) =\E_{n\leq L(N)}^{2\bin} f(T^{n}x) + \oh_{N\to\infty}(1).
\end{equation}
As shown in \cref{prop_2bin_ergodic_averages}, the conclusion of parts~\ref{itm_thm_erg_1} and~\ref{itm_thm_erg_2} of \cref{thm_ergodic} holds with $\frac{1}{N} \sum_{n = 1}^N f(T^{\vartheta(n)}x)$ replaced by $\E_{n\leq L(N)}^{2\bin} f(T^{n}x)$. Then by \eqref{eqn_dynamical_vartheta_to_2bin}, it also holds for $\frac{1}{N} \sum_{n = 1}^N f(T^{\vartheta(n)}x)$.

Finally, note that it also follows from \cref{thm_main} that
\begin{equation}
\label{eqn_dynamical_vartheta_to_Cesaro}
\mathbb{E}^{W}_{n\leq N}\,  f(T^{\vartheta(n)}x) = \frac{1}{N} \sum_{n = 1}^N f(T^{n}x) + \oh_{N\to\infty}(1).
\end{equation}
So part~\ref{itm_thm_erg_3} of \cref{thm_ergodic} follows from Birkhoff's pointwise ergodic theorem.
\end{proof}

\bibliographystyle{aomalphanomr}
\bibliography{mynewlibrary.bib}

\providecommand{\bysame}{\leavevmode\hbox to3em{\hrulefill}\thinspace}
\providecommand{\noopsort}[1]{}
\providecommand{\zbl}[1]{\href{http://www.zentralblatt-math.org/zmath/en/search/?q=an:#1}{Zbl~#1}}
\providecommand{\jfm}[1]{\href{http://www.emis.de/cgi-bin/JFM-item?#1}{JFM~#1}}
\providecommand{\arxiv}[1]{\href{http://www.arxiv.org/abs/#1}{arXiv~#1}}
\providecommand{\doi}[1]{\url{https://doi.org/#1}}
\providecommand{\href}[2]{#2}
\begin{thebibliography}{DMR09}

\bibitem[BKS19]{BKS19}
\bgroup\scshape{}V.~Bergelson\egroup{}, \bgroup\scshape{}G.~Kolesnik\egroup{}, and \bgroup\scshape{}Y.~Son\egroup{}, Uniform distribution of subpolynomial functions along primes and applications,  \emph{J. Anal. Math.} \textbf{137} no.~1 (2019), 135--187. \doi{10.1007/s11854-018-0068-1}.

\bibitem[BR22]{BR22}
\bgroup\scshape{}V.~Bergelson\egroup{} and \bgroup\scshape{}F.~Richter\egroup{}, Dynamical generalizations of the prime number theorem and disjointness of additive and multiplicative semigroup actions.,  \emph{Duke Math. J.} \textbf{171} no.~15 (2022), 3133--3200. \doi{10.1215/00127094-2022-0055}.

\bibitem[BM16]{BM16}
\bgroup\scshape{}V.~Bergelson\egroup{} and \bgroup\scshape{}J.~Moreira\egroup{}, Van der {C}orput's difference theorem: some modern developments,  \emph{Indag. Math. (N.S.)} \textbf{27} no.~2 (2016), 437--479. \doi{10.1016/j.indag.2015.10.014}.

\bibitem[BC00]{Boos}
\bgroup\scshape{}J.~Boos\egroup{} and \bgroup\scshape{}P.~Cass\egroup{}, \emph{Classical and Modern Methods in Summability}, Oxford University Press, 11 2000. \doi{10.1093/oso/9780198501657.001.0001}.

\bibitem[Bos81]{Boshernitzan81}
\bgroup\scshape{}M.~Boshernitzan\egroup{}, An extension of {H}ardy's class {$L$} of ``orders of infinity'',  \emph{J. Analyse Math.} \textbf{39} (1981), 235--255. \doi{10.1007/BF02803337}.

\bibitem[Bos82]{Boshernitzan82}
\bgroup\scshape{}M.~Boshernitzan\egroup{}, New ``orders of infinity'',  \emph{J. Analyse Math.} \textbf{41} (1982), 130--167. \doi{10.1007/BF02803397}.

\bibitem[Bos87]{Boshernitzan_preprint}
\bgroup\scshape{}M.~Boshernitzan\egroup{}, Uniform distribution, averaging methods and {H}ardy fields, unpublished, 1987.

\bibitem[Bos94]{Boshernitzan94}
\bgroup\scshape{}M.~Boshernitzan\egroup{}, Uniform distribution and {H}ardy fields,  \emph{J. Anal. Math.} \textbf{62} (1994), 225--240. \doi{10.1007/BF02835955}.

\bibitem[DMR09]{DMR09}
\bgroup\scshape{}M.~Drmota\egroup{}, \bgroup\scshape{}C.~Mauduit\egroup{}, and \bgroup\scshape{}J.~Rivat\egroup{}, Primes with an average sum of digits,  \emph{Compositio Mathematica} \textbf{145} no.~2 (2009), 271--292. \doi{10.1112/S0010437X08003898}.

\bibitem[Erd48]{Erdos48a}
\bgroup\scshape{}P.~Erd\H{o}s\egroup{}, On the integers having exactly {$K$} prime factors,  \emph{Ann. of Math. (2)} \textbf{49} (1948), 53--66. \doi{10.2307/1969113}.

\bibitem[FM05]{FM05}
\bgroup\scshape{}E.~Fouvry\egroup{} and \bgroup\scshape{}C.~Mauduit\egroup{}, Sur les entiers dont la somme des chiffres est moyenne,  \emph{Journal of Number Theory} \textbf{114} (2005), 135--152. \doi{10.1016/j.jnt.2005.03.007}.

\bibitem[Fra09]{Frantzikinakis09}
\bgroup\scshape{}N.~Frantzikinakis\egroup{}, Equidistribution of sparse sequences on nilmanifolds,  \emph{J. Anal. Math.} \textbf{109} (2009), 353--395. \doi{10.1007/s11854-009-0035-y}.

\bibitem[Gaj16]{Gajser16}
\bgroup\scshape{}D.~Gajser\egroup{}, On convergence of binomial means, and an application to finite markov chains,  \emph{Ars Math. Contemp.} \textbf{10} (2016), 393--410. Available at \url{https://api.semanticscholar.org/CorpusID:55133538}.

\bibitem[GK91]{GK91}
\bgroup\scshape{}S.~W. Graham\egroup{} and \bgroup\scshape{}G.~Kolesnik\egroup{}, \emph{van der {C}orput's method of exponential sums}, \emph{London Mathematical Society Lecture Note Series} \textbf{126}, Cambridge University Press, Cambridge, 1991. \doi{10.1017/CBO9780511661976}.

\bibitem[Har49]{H49}
\bgroup\scshape{}G.~H. Hardy\egroup{}, \emph{Divergent Series}, Clarendon Press, Oxford, 1949, Reprinted by AMS Chelsea Publishing, Providence, RI, 1991/2000.

\bibitem[Loy23]{LOYD23}
\bgroup\scshape{}K.~Loyd\egroup{}, A dynamical approach to the asymptotic behavior of the sequence $\omega (n)$,  \emph{Ergodic Theory and Dynamical Systems} \textbf{43} no.~11 (2023), 3685–3706. \doi{10.1017/etds.2022.81}.

\bibitem[LM25]{LM25}
\bgroup\scshape{}K.~Loyd\egroup{} and \bgroup\scshape{}S.~Mondal\egroup{}, Ergodic averages along sequences of slow growth,  \emph{Journal of the London Mathematical Society} \textbf{111} no.~3 (2025), e70124. \doi{https://doi.org/10.1112/jlms.70124}.

\bibitem[Rei26]{reilly26}
\bgroup\scshape{}M.~Reilly\egroup{}, Uniform weighted averages and a conjecture of {B}ergelson, {M}oreira, and {R}ichter, 2026. Available at \url{https://arxiv.org/abs/2602.20606}.

\bibitem[SF14]{Spencer14}
\bgroup\scshape{}J.~Spencer\egroup{} and \bgroup\scshape{}L.~Florescu\egroup{}, \emph{Asymptopia}, \emph{Student mathematical library}, American Mathematical Society, 2014. Available at \url{https://books.google.com/books?id=g20CkAEACAAJ}.

\end{thebibliography}



\bigskip
\footnotesize
\noindent
Vitaly Bergelson\\
\textsc{The Ohio State University}\\
\href{mailto:vitaly@math.ohio-state.edu}
{\texttt{vitaly@math.ohio-state.edu}}

\bigskip
\noindent
Michael Reilly\\
\textsc{The Ohio State University}\\
\href{mailto:reilly.201@osu.edu}
{\texttt{reilly.201@osu.edu}}

\bigskip
\noindent
Florian K.\ Richter\\
\textsc{{\'E}cole Polytechnique F{\'e}d{\'e}rale de Lausanne (EPFL)}\\
\href{mailto:f.richter@epfl.ch}
{\texttt{f.richter@epfl.ch}}

\end{document}